\documentclass{amsart}
\usepackage{amsthm, graphicx, graphics, latexsym}
\usepackage{psfrag}
\usepackage{multirow}

\newtheorem{theorem}{Theorem}[section]
\newtheorem{example}{Example}[section]
\newtheorem{proposition}{Proposition}[section]
\newtheorem{definition}[theorem]{Definition}

\newtheorem{conjecture}{Conjecture}[section]
\newtheorem{lemma}{Lemma}[section]
\newtheorem{corollary}{Corollary}[section]

\title{On the shape of subword complexity sequences of finite words}
\author{Hannah Vogel\textsuperscript{1}}

\begin{document}
\maketitle
\footnotetext[1]{Carnegie Mellon University, Masters Thesis}


\section{Introduction}

Let $\mathcal{A}$ be a set. We will call $\mathcal{A}$  an alphabet and the elements of $\mathcal{A}$ letters. An alphabet $\mathcal{A}$ is finite if it has a finite number of letters. We will only be considering finite alphabets. By $\mathcal{A}_k$ we denote an alphabet with $k$ letters.

\begin{definition}
A word $w$ over an alphabet $\mathcal{A}$ is a sequence of letters $w=w_1 w_2 \ldots w_n$ such that $ w_i \in \mathcal{A}$ for all $i=1,\ldots,n$. The reverse word, denoted by $w^{-1}$, is $w^{-1}=w_n w_{n-1}\ldots w_1$. The length of $w$, denoted by $\left|w\right|$, is $n$.
\end{definition}

If $u$ and $v$ are words, $uv$ denotes their concatenation. For a positive integer $n$, $u^n= uuu...u$ ($n$ times). By $\epsilon$ we denote the only word of length 0, the \textit{empty word}. $\mathcal{A}^n$ is the set of all words over $\mathcal{A}$ of length $n$. $\mathcal{A}^+ = \cup_{n \ge 1} \mathcal{A}^n $ is a free semigroup (an associative set with a binary operation) over $\mathcal{A}$ with the group operation being string concatenation. We define $\mathcal{A}^*=\mathcal{A}^+ \cup \{ \epsilon \}$. 

\begin{definition}
A word $u$ is a subword of a word $w$ if there exist $p, q \in \mathcal{A}^*$ such that $w = puq$. Equivalently, we say that a word $u$ is a subword (factor) of a word $w=w_1w_2\ldots w_n$ if there exist integers $i, j \in \mathbb{N}$ with $0 < i \le j$ such that $u=w_i w_{i+1} \ldots w_j$. We denote this occurrence of $u$ in $w$ by $w[i,j]$.
\end{definition}

\begin{example}
If $w=0110101110$, then $u=1010$ is a subword of $w$, but $v=1001$ is not a subword of $w$.
\end{example}

\begin{definition}
Let $w$ be a word. We define $Sub_w(m)$ to be the set of all subwords of length $m$ of $w$. We define $Sub(w)$ to be the set of all subwords of $w$. \\
\end{definition}

\begin{example}
Let $w=2110$. Then
\begin{align*}
 Sub_w(0) & = \{ \epsilon \}; \\
 Sub_w(1) & = \{0, 1, 2\}; \\
 Sub_w(2) & = \{10, 11, 21\}; \\
 Sub_w(3) & = \{110, 211\}; \\
 Sub_w(4) & = \{ 2110 \}; \\ 
 Sub(w) & = \{ \epsilon, 0, 1, 2, 10, 11, 21, 110, 211, 2110\}. 
\end{align*}
\end{example}

\begin{definition}
An integer $p \ge1$ is a period of a word $w=a_1a_2 \ldots a_n$, where $a_i \in \mathcal{A}$, if $a_i = a_{i+p}$ for $i=1 \ldots n-p$. If no such $p$ exists then we say $w$ is aperiodic.
\end{definition}

In this paper, we will be focusing on the subword complexity sequences of words. The subword complexity of a word $w$ is a function that assigns for each positive integer $n$,  the number of distinct subwords of length $n$ in $w$, $p_w(n)$.

\begin{definition}
Given a word $w$ of length $N$ over $\mathcal{A}_k$, the subword complexity function of $w$, $p_w(n)$, counts the number of distinct subwords of length $n$ in $w$. The subword complexity sequence of $w$ is the sequence $p_w=(p_w(1), p_w(2), \ldots, p_w(N))$.
\label{subcomp:def}
\end{definition}

\begin{example}
Let $w=01101$. Then
\begin{align*}
& Sub_w(1) = \{ 0, 1 \} & \left|Sub_w(1)\right| = 2 \\
& Sub_w(2) = \{01, 10, 11 \} & \left|Sub_w(2)\right| = 3 \\
& Sub_w(3) = \{ 011, 101, 110 \} & \left|Sub_w(3)\right| = 3 \\
& Sub_w(4) = \{0110, 1101\} & \left|Sub_w(4)\right| = 2 \\
& Sub_w(5) = \{ 01101 \} & \left|Sub_w(5)\right| = 1
\end{align*}

\par so $p_w = (2, 3, 3, 2, 1)$.\\
\end{example}

The subword complexity of a word is a good measure of the randomness of the word. The randomness of a word is dependent not only on the number of distinct letters in the word, but also how they are positioned. For example, periodic words are of low randomness, and have low subword complexity. Aperiodic words have higher randomness, and subword complexity, than periodic words, but there are varying degrees of randomness in aperiodic words. The shape of the subword complexity sequence of a word gives insight to what the word itself looks like. For example, consider $$p_w= (1,1,1,1,1,1).$$ Then, without loss of generality, we know the word of length $6$ is $w=000000$.

There are no known necessary and sufficient conditions for which sequences of numbers are subword complexity sequences. For a list of necessary conditions and a list of sufficient conditions, see Ferenczi \cite{F:1999}. Subword complexity sequences of finite and infinite words have become an important area of research in the combinatorics of words. Applications of subword complexity sequences include Computer Science, Algebra, and Biology.

We will restrict our attention mainly to subword complexity sequences of finite words. In Section 3 we will discuss subword complexity sequences in more detail. For further reading on subword complexity, see Anisiu and Cassaigne \cite{AC:2004}, Allouche \cite{A:1994}. Section 4 will discuss de Bruijn words, which have maximal subword complexity. For more on de Bruijn words, see Anisiu, Blazsik and Kasa \cite{ABK:2010}, Chan, Games and Key \cite{CGK:1982}, Matou\v{s}ek and Ne\v{s}et\v{r}il \cite{MN:2009}. In Sections 5 and 6 we will discuss Sturmian words, which have minimal subword complexity for non-ultimately periodic words. For further reading on Sturmian words, see Allouche and Shallit \cite{AS:2003}, de Luca \cite{AL:1997}, de Luca and de Luca \cite{dLdL:1994}, Fogg \cite{PF:2002}, Lothaire \cite{L:2002}, Matom\"aki and Saari \cite{KS:2012}, Vuillon \cite{V:2003}.

\bigskip

\section{Preliminaries}

\begin{definition}
For two words $u$ and $w$ in $\mathcal{A}^*$, we say that $u$ is a prefix of $w$ if there exists a word $q$ such that $w=uq$. We denote the set of prefixes of a word $w$ by $Pref_w$.
\end{definition}

\begin{definition}
For two words $u$ and $w$ in $\mathcal{A}^*$, we say that $u$ is a suffix of $w$ if there exists a word $p$ such that $w=pu$. We denote the set of suffixes of a word $w$ by $Suf_w$.
\end{definition}

\begin{example}
Let $w=0110101110$. Then $u=01110$ is a suffix of $w$ of length 5, and $v=011010$ is a prefix of $w$ of length 6.
\end{example}

\begin{definition}
The multiplicity of a subword $u$ of $w$ is the number of occurrences of  $u$ in $w$.
\end{definition}

\begin{example}
Let $w=01101100$. Then the multiplicity of $u=0110$ in $w$ is $2$. The two occurrences of $u$ in $w$ are $w[1,4]$ and $w[4,7]$.
\end{example}

For a subword $u$ of $w$ we consider the maximal subset $R_u$ of $\mathcal{A}$ such that
\[
uR_u \subseteq Sub(w)
\]
so that $u$ occurs in $w$ followed on the right by any one of the letters in $R_u$, and only by letters in $R_u$. 

In a symmetric way we can define the left maximal subset $L_u$ of $\mathcal{A}$ such that 
\[
L_uu \subseteq Sub(w)
\]
so that $u$ occurs in $w$ preceded on the left by any one of the letters in $L_u$ and only by letters in $L_u$. 

\begin{definition}
 A subword $u$ has \text{valence }$k$ if it can be extended on the right in $w$ by exactly $k$ distinct letters. 
\end{definition}

Note that $\left|R_u\right|$ is the valence of $u$ in $w$.

\begin{example}
Let $w=1211210121122$, and consider $u=121$. We have that

\begin{align*}
&R_u = \{0, 1\} \\
&uR_u = \{1210, 1211\}
\end{align*}

the multiplicity of $u$ in $w$ is 3, but the valence of $u$ in $w$ is $\left|R_u\right|= 2$.

\end{example}

 In this example  $u$ has valence $2$ in $w$ because it can be extended on the right by the two letters $0$ and $1$. The subword $v=122$ in the example above has valence $0$ because it is not followed by any letters in $w$.

\begin{definition}
A nonempty subword $u$ of $w$ is called \textit{special} if it has valence $\ge 2$. This implies that there exist at least two letter $p,q \in \mathcal{A}$, $p \neq q$ such that $up, uq \in Sub(w)$. So all special subwords of a word have a valence of at least $2$. 
\end{definition}

\begin{example}
Let $w= 011010$. Then the set of special subwords of $w$ is $\{1, 01\}$. \\
\end{example}

\begin{definition}
Let $s_w(n,i)$, or just $s(n,i)$ when there is no ambiguity, be the number of distinct subwords of length $n$ in $w$ that have valence $i$. Let $K_w$ be the minimal length of a suffix that occurs only once in $w$. Let $R_w$ be the minimum $n$ such that all the subwords of $w$ of length $n$ have valence at most $1$.
\end{definition}

\begin{example}
Let $w=101100$. Then
\[Sub(w) = \{ \epsilon, 0,1, 00, 01, 10, 11, 100, 101, 110, 011, 0110, 1011, 1100, 01011, 10110, 101100\}\]
The subword structure of $w$ can be represented as a tree in which the subwords of $w$ are the tree nodes, and there is an edge between a "parent" node, a subword of length $n$, and a "child" node, a subword of length $n+1$, if the parent subword is a prefix of the child subword in $w$ :

 \begin{figure}[ht]
\begin{center}
\psfragscanon
\psfrag{e}{$\epsilon$}
\psfrag{0 }{$0$}
\psfrag{1}{$1$}
\psfrag{00}{$00$}
\psfrag{01}{$01$}
\psfrag{10}{$10$}
\psfrag{11}{$11$}
\psfrag{011}{$011$}
\psfrag{100}{$100$}
\psfrag{101}{$101$}
\psfrag{110}{$110$}
\psfrag{0110}{$0110$}
\psfrag{1011}{$1011$}
\psfrag{1100}{$1100$}
\psfrag{01100}{$01100$}
\psfrag{10110}{$10110$}
\psfrag{101100}{$101100$}
\includegraphics[scale=1]{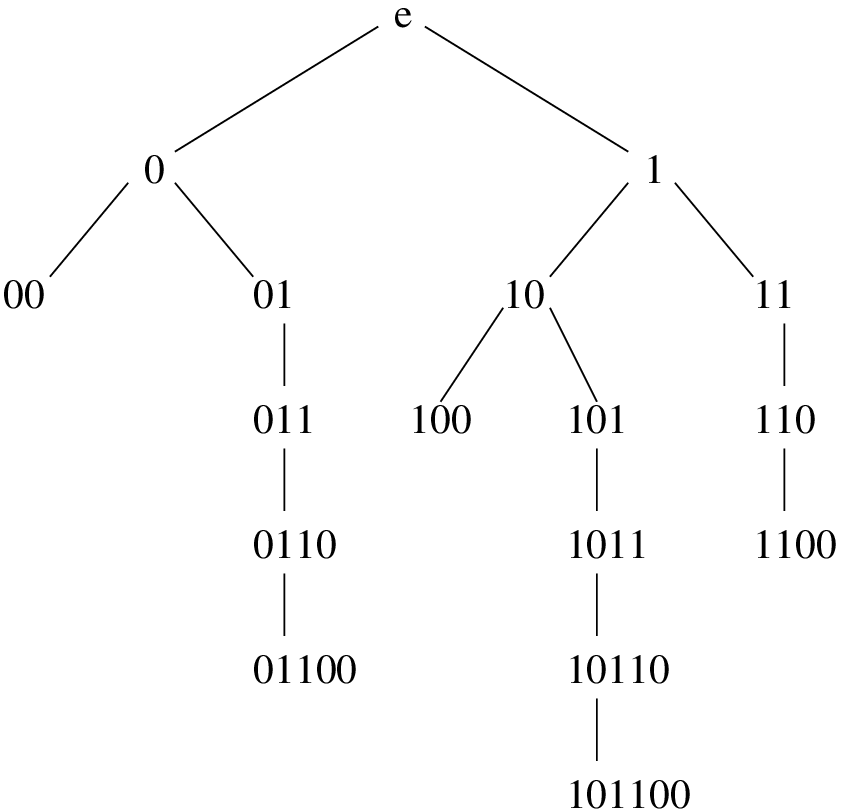}
\end{center}
\end{figure}

So for example, 0 is a parent of 00 and 01. 
We consider $\epsilon$  to be in the zeroth row. We can see that $R_w = 3$ because the third row is the earliest row where each parent has at most one child. This is the same as saying that every subword of length 3 has multiplicity 1. Now $K_w =2$, which we can see because in the second row, $00$ does not have any descendants, and this corresponds to the suffix of $w$ having multiplicity 1.
\end{example}

Note that for any word $w$, $s(n,0)=0$ for $0 \le n \le K_w -1$. This is because any subword of $w$ of length $ < K_w$ is followed by at least one letter in $w$. Also, $s(n,0)=1$ for $K_w \le n \le \left|w\right|$, because for any $n$, $K_w \le n \le \left|w\right|$, only one subword of $w$ of length $n$ is not followed by any letter in $w$, the suffix of $w$ of length $n$.

Thus we have that $s(n,0) \le 1$ for all $0 \le n \le \left|w\right|$ and the number of subwords of $w$ that are not followed by any letter in $w$ is $\left|w\right| - K_w + 1$.

Note that $s_w(\left|w\right| - 1,2) = s_w(\left|w\right|,2) = 0$ and thus $R_w$ is always defined for finite words.


\bigskip
\section{Subword Complexity}

Recall from Definition \ref{subcomp:def} that a subword complexity function of $w$, $p_w(n)$, counts the number of distinct subwords of length $n$ in $w$. There are several interesting open problems involving the subword complexity of finite words:
\begin{itemize}
\item Which finite sequences of natural numbers are subword complexity sequences?
\item How many distinct subword complexity sequences of words of length $N$ over $\mathcal{A}_k$ are there? 
\item How many words of length $N$ over $\mathcal{A}_k$ have exactly $m$ distinct subwords of length $l$?
\end{itemize} 

As we mentioned earlier, the subword complexity of a word is of interest because in some sense it measures the randomness of a word. 
 
In this section we give several necessary conditions for a function to be a subword complexity function (Propositions \ref{3.1}, \ref{card:prop}, Theorems \ref{min:th}, \ref{mineq:th}, \ref{main}, and \ref{kineq:th}). Additional information about the shape of a subword complexity function is given by Propositions \ref{Nmax:prop} and \ref{R+K}.
 
\begin{example}
The subword complexity sequences of all binary words of length 3:
\begin{align*}
& w= 000 & p_w= (1,1,1) \\
& w= 100 & p_w= (2,2,1) \\
& w= 101 & p_w= (2,2,1) \\
& w= 110 & p_w= (2,2,1) \\
& w= 011 & p_w= (2,2,1) \\
& w= 010 & p_w= (2,2,1) \\
& w= 001 & p_w= (2,2,1) \\
& w= 111 & p_w= (1,1,1) 
\end{align*}

Note that there are only two distinct subword complexity sequences of binary words of length 3.
\end{example}

\begin{proposition}
Let $w$ be a word. Then for integers $m, n \ge 0$ we have that $p_w(m+n) \le p_w(m)p_w(n)$.
\label{3.1}
\end{proposition}

\begin{proof}
We can express every subword of length $m+n$ as a subword of length $m$ followed by a subword of length $n$, so there are at most $p_w(m)p_w(n)$ subwords of length $n+m$.
\end{proof} 

\begin{corollary}
Let $w$ be a word of length $N$ over $\mathcal{A}_k$. Then for $1 \le n \le N-1$, $p_w(n+1) \le k\,p_w(n)$.
\end{corollary}

\begin{proposition}
Let $w$ be a word over $\mathcal{A}_k$. Then $1 \le p_w(n) \le k^n$ for $1 \le n \le \left|w\right|$. 
\label{card:prop}
\end{proposition}

\begin{proof}
There are only $k^n$ words of length $n$ over $\mathcal{A}_k$, thus$1 \le p_w(n) \le k^n$ for $1 \le n \le \left|w\right|$. 
\end{proof}

So if we consider the binary alphabet $\mathcal{A}_2 = \{0,1 \}$, then $p_w(n) \le 2^n$ for all $n$. 

\begin{theorem}
Let $w$ be a word over the alphabet $\mathcal{A}_k$, $\left|w\right| = N$. Then $p_w(n) \le \min\{k^n, N-n+1\}$ for $1 \le n \le N$.
\label{min:th}
\end{theorem}
\begin{proof}
By Proposition 3.2 we have that $p_w(n) \le k^n$. We need to show that $p_w(n) \le N-n+1$. Let $w$ be a word of length $N$. Then there are $N-n+1$ not necessarily distinct subwords (contiguous blocks of letters) of length $n$ in $w$. So we have at most $N-n+1$ distinct subwords of length $n$. Thus $p_w(n) \le N-n+1$, so $p_w(n) \le \min\{k^n, N-n+1\}$.\\
\end{proof}

The following theorem is going to be proved in Section 4.

\begin{theorem}
For every $n \ge 0$, there exists a word $w$ of length $N$ over $\mathcal{A}_k$ with subword complexity $p_w(n)= \min\{k^n, N-n+1\}$ for $1 \le n \le N$. It follows that for every word $u$ of length $N$ over $\mathcal{A}_k$ we have that $p_u(n) \le p_w(n)$ for all $1 \le n \le N$.
\label{mineq:th}
\end{theorem}

\begin{definition}
A word of length $N$ over $\mathcal{A}_k$ with subword complexity $p_w(n)= \min\{k^n, N-n+1\}$ for $1 \le n \le N$ is called a de Bruijn word. 
\label{deBdef}
\end{definition}

De Bruijn words are often used in decoding problems. We will discuss de Bruijn words extensively in Section 4.

\begin{definition}
A sequence $(s_1, s_2, \ldots, s_n)$ is unimodal if there exists $t$, $1 \le t \le n$, such that 
$
s_1 \le s_2 \le \ldots \le s_t
$ 
and
$
s_t \ge s_{t+1} \ge \ldots \ge s_n .
$
\end{definition}

\begin{example}
Consider all words of length $1 \le n \le 7$ over $\mathcal{A}_2$. The different subword complexity sequences associated to these lengths are:\\

\end{example}
\smallskip
\begin{center}
\begin{tabular}{|c|c|}
\hline
\multicolumn{2}{|c|}{Subword complexity sequences of binary words of length $n$} \\
\hline
$n$ & $p_w$  \\ \hline\hline
\multirow{1}{*}{1} & (1) \\ \hline
\multirow{2}{*}{2} & (1,1) \\
 & (2,1)  \\ \hline
\multirow{2}{*}{3} & (1,1,1) \\
 & (2,2,1) \\ \hline
 \multirow{3}{*}{4}  & (1,1,1,1) \\ 
 & (2,2,2,1) \\ & (2,3,2,1) \\ \hline 
 \multirow{4}{*}{5} & (1,1,1,1,1) \\ & (2,2,2,2,1) \\ & (2,3,3,2,1) \\ & (2,4,3,2,1) \\ \hline
\multirow{5}{*}{6}  & (1,1,1,1,1,1) \\ & (2,2,2,2,2,1)  \\ & (2,3,3,3,2,1)  \\ & (2,3,4,3,2,1)   \\ & (2,4,4,3,2,1) \\ \hline
\multirow{7}{*}{7} & (1,1,1,1,1,1,1)  \\ & (2,2,2,2,2,2,1)  \\ & (2,3,3,3,3,2,1)  \\ & (2,3,4,4,3,2,1)   \\ & (2,3,5,4,3,2,1) \\ & (2,4,4,4,3,2,1)  \\ & (2,4,5,4,3,2,1) \\ \hline
\end{tabular}\\
\end{center}
\bigskip

Looking at these subword complexity sequences for words of lengths $1$ to $7$ we see that they are all unimodal. In the following theorem, we will prove that, in fact, all subword complexity sequences are unimodal. The statement of the following theorem is taken from de Luca \cite{AL:1999}.

\begin{theorem}
	The subword complexity sequence of a finite word over $\mathcal{A}_k$ is unimodal. Moreover, once it starts decreasing, it decreases by 1 until it reaches 1.
\label{main}	
\end{theorem}

\begin{proof}
Let $w$ be a finite word of length $N$, and let $p_w$ be its subword complexity sequence. Recall that the parameter $R_w$ is the minimum $n$ such that all the subwords of $w$ of length $n$ are followed by at most one letter in $w$, and the parameter $K_w$ is the minimal length of a suffix of $w$ that occurs only once in $w$.
 
Note that for $n <K_w$, 
\begin{equation}
p_w(n+1) = p_w(n) + \sum_{i=2}^k (i-1)s_w(n,i) \ge p_w(n)+s_w(n),
\label{eq1}
\end{equation}
 where $s_w(n,i)$ is the number of subwords of length $n$ followed by $i$ distinct letters in $w$, and  $$s_w(n)= \sum_{i=2}^k s_w(n,i).$$ For $n \ge K_w$, 
 \begin{equation}
 p_w(n+1) = p_w(n)+\sum_{i=2}^k (i-1)s_w(n,i)-1 \ge p_w(n) + s_w(n) -1.
 \label{eq2}
 \end{equation}
 Let $m=\min\{R_w, K_w\}$ and $M = \max\{R_w, K_w\}$. We will show that $p_w$ is strictly increasing on the interval $[1, m]$, nondecreasing on the interval $[m, M]$, and decreasing on the interval $[M, N]$. Also, if $R_w < K_w$ then $p_w$ is constant on the interval $[m,M]$. We will consider two cases:

\begin{enumerate}
\item $R_w < K_w$
\par For $n < R_w$, we have that $s_w(n) \ge 1$, so by equation (\ref{eq1}) $$p_w(n+1) \ge p_w(n) + s_w(n) \ge p_w(n)+1,$$ hence $p_w$ is strictly increasing on the interval $[1,R_w]$. 

For $R_w \le n< K_w$, we have that $s_w(n,i)=0$ for $i \ge 2$, by equation (\ref{eq1}) $$p_w(n+1)=p_w(n),$$ hence $p_w$ is constant on the interval $[R_w, K_w]$.  

For $K_w \le n< N$, we have that $s_w(n,i)=0$ for $i \ge 2$, by equation (\ref{eq2}) $$p_w(n+1)=p_w(n) -1,$$ so $p_w$ is strictly decreasing on the interval $[K_w, N]$. \\

\begin{figure}[ht]
\begin{center}
\psfragscanon
\psfrag{1}{$1$}
\psfrag{R}{$R_w$}
\psfrag{K}{$K_w$}
\psfrag{N}{$N$}
\psfrag{p}{$p_w(n)$}
\psfrag{n}{$n$}
\includegraphics[scale=.45]{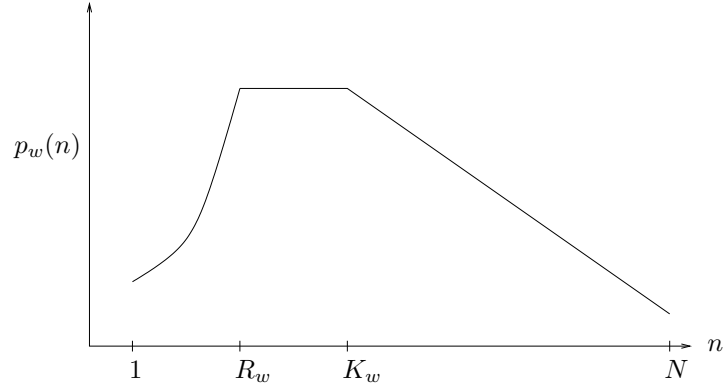}
\caption{The subword complexity function when $R_w < K_w$.} 
\end{center}
\end{figure}

\item $K_w \le R_w$
\par For $n < K_w$, we have that $s_w(n) \ge 1$, by equation (\ref{eq1}) $$p_w(n+1) \ge p_w(n) + s_w(n) \ge p_w(n)+1,$$ hence $p_w$ is strictly increasing on the interval $[1,K_w]$.

 For $K_w \le n< R_w$, we have that $s_w(n)\ge1$, by equation (\ref{eq2}) $$p_w(n+1)=p_w(n)+s_w(n)-1\ge p_w(n),$$ hence $p_w$ is nondecreasing on the interval $[K_w, R_w]$.  
 
 For $R_w \le n< N$, we have that $s_w(n,i)=0$ for $i \ge 2$, by equation (\ref{eq2}) $$p_w(n+1)=p_w(n) -1,$$ so $p_w$ is strictly decreasing on the interval $[R_w, N]$. \\

\begin{figure}[ht]
\begin{center}
\psfragscanon
\psfrag{1}{$1$}
\psfrag{R}{$K_w$}
\psfrag{K}{$R_w$}
\psfrag{N}{$N$}
\psfrag{p}{$p_w(n)$}
\psfrag{n}{$n$}
\includegraphics[scale=.45]{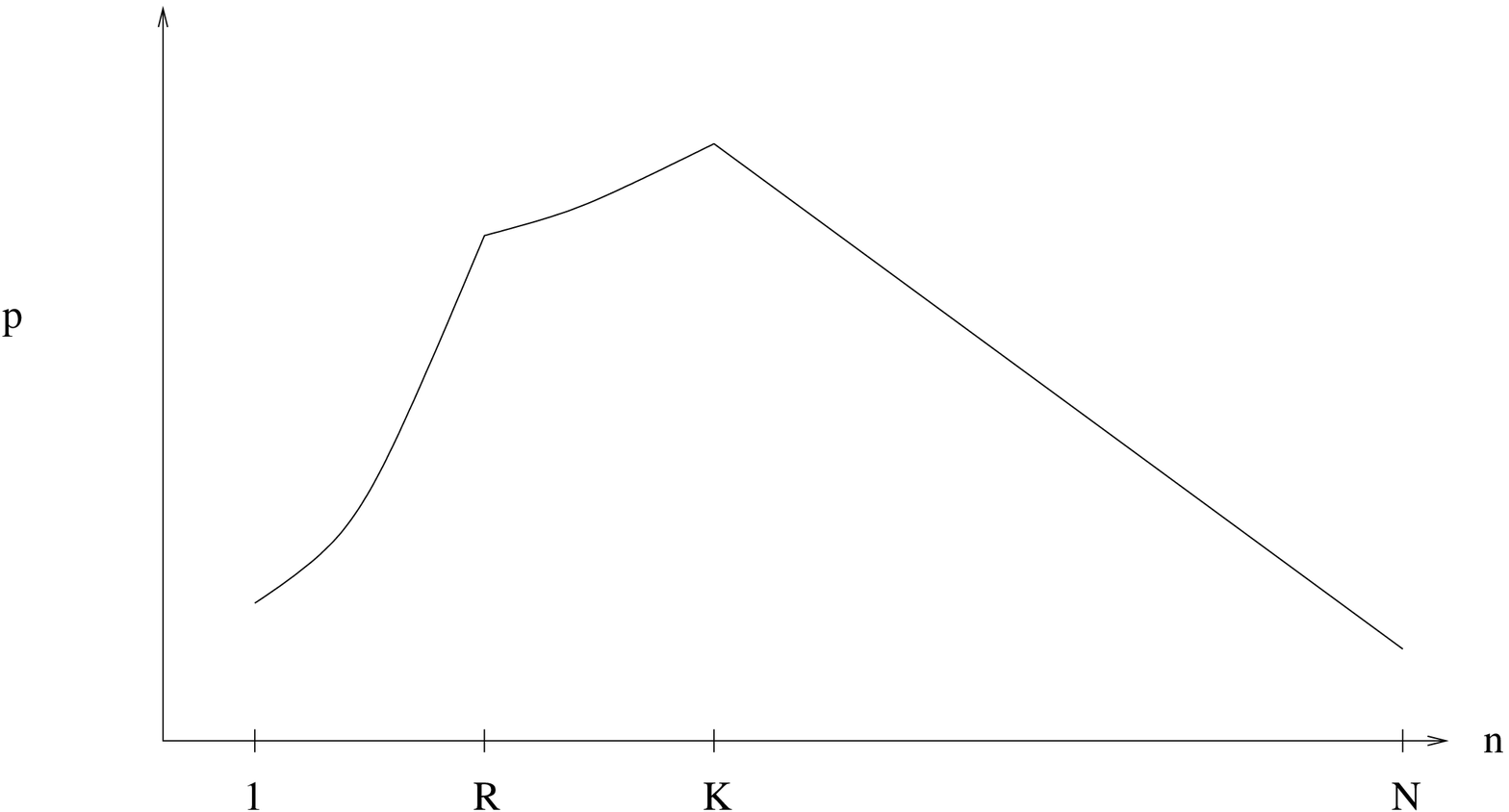}
\caption{The subword complexity function when $K_w \le R_w$.} 
\end{center}
\end{figure}

\end{enumerate}

Note that $p_w(n+1) = p_w(n) - 1$ for $n \ge M$, and that $p_w(N)=1$.\\
\end{proof}

The following theorem, originally proved for infinite words in Allouche and Shallit \cite{AS:2003}, has been modified to apply to finite words.

\begin{theorem}
Let $w$ be a word over an alphabet $\mathcal{A}_k$ and $\left|w\right|=N$. Then for all $1 \le n < K_w$,
\[ 
p_w(n+1) - p_w(n) \le k(p_w(n)- p_w(n-1)).
\]
\label{kineq:th}
\end{theorem} 

\begin{proof}
Let $n$ be an integer,$1 \le n < K_w$.

For $1 \le i \le k$, let $T(n,i)$ be the set of subwords of $w$ of length $n$ with valency \textit{at least} $i$, that is the set of subwords of $w$ of length $n$ that are followed by at least $i$ distinct letters of $\mathcal{A}_k$. Let $t(n,i)= \left|T(n,i)\right|$.

Let $u \in T(n,i)$. The suffix of length $n-1$ of $u$ is an element of $T(n-1,i)$. Thus $T(n,i) \subseteq \mathcal{A}_k T(n-1,i)$. So $$t(n,i) = \left|T(n,i)\right| \le \left|\mathcal{A}_k\right| \left|T(n-1,i)\right| = k\,t(n-1,i).$$ 

Note that since $n < K_w$, every subword of length $n$ of $w$ is a the prefix of at least one subword of length $n+1$ of $w$ and $$p_w(n+1)-p_w(n)= t(n,2) + t(n,3) + \ldots + t(n,k).$$

Similarly $p_w(n)-p_w(n-1)= t(n-1,2) + t(n-1,3) + \ldots + t(n-1,k).$
 
Since  $t(n,i) \le kt(n-1,i)$,

\begin{align*}
p_w(n+1)-p_w(n) &= t(n,2) + \ldots + t(n,k) \le k \,t(n-1,2) + \ldots + k\,t(n-1,k)\\
& = k(p_w(n)-p_w(n-1)).
\end{align*}

\end{proof}

The proofs of the following two theorems proved by de Luca in \cite{AL:1999} have been modified.

\begin{proposition}
The subword complexity of a word $w$ of length $N$ takes its maximal value at $R_w$ and, moreover, $p_w(R_w) = N - \max\{R_w,K_w\} +1$.
\label{Nmax:prop}
\end{proposition}

\begin{proof}
If $R_w \ge K_w$ then by Theorem \ref{main}, $p_w$ takes its maximal value in $R_w$.  If $R_w < K_w$ then again by Theorem \ref{main} we have that $p_w$ takes its maximal value at $K_w$, but because $p_w$ is constant in the interval $[R_w, K_w]$, $p_w(R_w) = p_w(K_w)$ and so $p_w$ reaches its maximal value in $R_w$.

If $w = a^N$ for some $a \in \mathcal{A}$ then we have that $R_w = 1$, $K_w = N$ and $p_w(R_w) = p_w(1) = N - N+1 = 1$ which we know to be true since $p_w(n)=1$ for every $ n \in [1,N]$.

Assume that $w$ contains at least two letters. If $R_w \ge K_w$ then for every $n \in [R_w, N-1]$, $p_w(n+1)=p_w(n)-1$. This implies that $1 = p_w(n) - p_w(n+1)$. Continuing this recursion we have that $1 = p_w(N) = p_w(R_w) - (N-R_w)$ and so $p_w(R_w) = N - R_w +1$. If $R_w < K_w$ then for every $n \in [K_w, N-1]$, $p_w(n+1)=p_w(n)-1$ and we can derive that $1 = p_w(N) = p_w(K_w) - (N-K_w)$. Since $p_w(R_w)=p_w(K_w)$, the result follows.\\
\end{proof}

\begin{proposition}
Let $w$ be a word of length $N$ over $\mathcal{A}_k$. If $R_w=1$ then $R_w+K_w=N+1$. Otherwise $N \ge R_w + K_w$.
\label{R+K}
\end{proposition}

\begin{proof}
If $R_w=1$ then $w=a^N$ for some $a \in \mathcal{A}_k$. So $K_w=N$ and it follows that $R_w+K_w = N+1$. 

If $K_w=1$ then $w=ua$, where $a \in \mathcal{A}_k$ and $a$ does not occur in $u$. Then $R_w \le N-1$, equality holding when $w=b^{N-1}a$, $a \neq b$. So $R_w+K_w \le N$.
 
Assume $R_w,K_w > 1$ and let $m = \min\{R_w,K_w\}$. For all $n \in [1, m-1]$, $s(n,1)=0$, so $p_w(n+1) \ge p_w(n) +1$. 
Suppose that $R_w \le K_w$. Then $m = R_w$ and since $p_w(1) \ge2$, $p_w(R_w) \ge R_w +1$. From Proposition \ref{Nmax:prop} we have that $p_w(R_w) = p_w(K_w) = N -K_w + 1$. So $N - K_w +1 = p_w(R_w) \ge R_w +1 $ and thus $N \ge R_w + K_w$. 
If $R_w > K_w$  we have that $m = K_w$ and $p_w(K_w) \ge K_w +1$. Again by Proposition \ref{Nmax:prop} we have $p_w(K_w) \le p_w(R_w) = N - R_w +1$ and thus $N \ge R_w + K_w$. 
\end{proof}


\bigskip
\section{de Bruijn Words}

Recall from the previous section that a de Bruijn word is a word of length $N$ over $\mathcal{A}_k$ with subword complexity $p_w(n)= \min\{k^n, N-n+1\}$ for $1 \le n \le N$. This is a generalization of the traditional definition of the de Bruijn words, which arise from de Bruijn graphs, and can be thought of as the shortest random-like words.\\

\begin{definition}
A directed graph $G$ is a pair $(V,E)$, where $E$ is a subset of the cartesian product $V \times V$. The ordered pairs $(v,w) \in E$ are called directed edges.
\end{definition}

\begin{figure}[ht]
\begin{center}
\psfragscanon
\psfrag{u}{$v_1$}
\psfrag{v}{$v_4$}
\psfrag{w}{$v_2$}
\psfrag{z}{$v_3$}
\psfrag{e1}{$e_1$}
\psfrag{e2}{$e_2$}
\psfrag{e3}{$e_3$}
\psfrag{e4}{$e_4$}
\psfrag{e5}{$e_5$}
\includegraphics[scale=1]{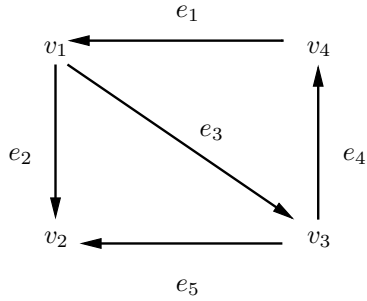}
\caption{A directed graph on four vertices.} 
\label{dir:line}
\end{center}
\end{figure}

The number of directed edges ending in $v$ is the in-degree of $v$ and is denoted by $deg^+(v)$. Similarly, the number of directed edges originating in $v$ is the out-degree of $v$ and is denoted by $deg^-(v)$. In the figure above, $deg^+(v_3) = 1$ and $deg^-(v_3)=2$.  

We say that a directed edge $e=(v,w)$ has source $v$, denoted by $s(e)$, and  target $w$, denoted by $t(e)$.

\begin{definition}
A directed walk in a directed graph $G= (V,E)$ is a sequence $(v_0, e_1, v_1, e_2, \ldots, e_n, v_n)$ such that $e_i = (v_{i-1}, v_i) \in E$ for each $i = 1, 2, \ldots, n$.
\end{definition}

\begin{definition}
The de Bruijn graph $B_k(n) = (V,E)$ is a directed graph whose vertices are words of length $n$ over $\mathcal{A}_k$ and whose edges are words of length $n+1$ over $\mathcal{A}_k$ such that a word $w$ of length $n+1$ is a directed edge from vertex $v_1$, which is the prefix of length $n$ of $w$, to vertex $v_2$, which is the suffix of length $n$ of $w$.
\end{definition}

Note that a directed walk of length $m$ in $B_k(n)$ from vertex $v$ to vertex $w$ corresponds to a word of length $n+m$ with prefix $v$ and suffix $w$. Therefore there is a correspondence between words over $\mathcal{A}_k$ of length at least $n$ and walks in $B_k(n)$. Let $u,v \in E(B_k(n))$. Then the word $uv$ is a walk in $B_k(n)$ from $u$ to $v$. This proves that $B_k(n)$ is a connected directed graph.

\begin{figure}[ht]
\begin{center}
\psfragscanon
\psfrag{0}{$0$}
\psfrag{1}{$1$}
\includegraphics[scale=.75]{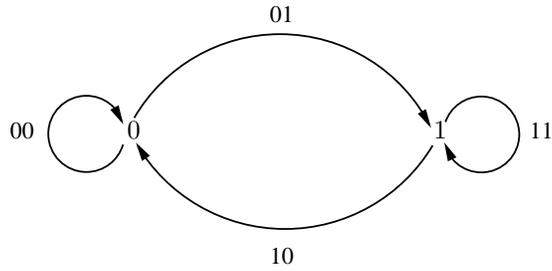}
\end{center}
\caption{ The de Bruijn graph $B_2(1)$.}
\end{figure}

\begin{figure}[ht]
\begin{center}
\psfragscanon
\psfrag{00}{$00$}
\psfrag{11}{$11$}
\psfrag{01}{$01$}
\psfrag{10}{$10$}
\psfrag{100}{$100$}
\psfrag{110}{$110$}
\psfrag{001}{$001$}
\psfrag{011}{$011$}
\psfrag{000}{$000$}
\psfrag{111}{$111$}
\psfrag{010}{$010$}
\psfrag{101}{$101$}
\includegraphics[scale=.75]{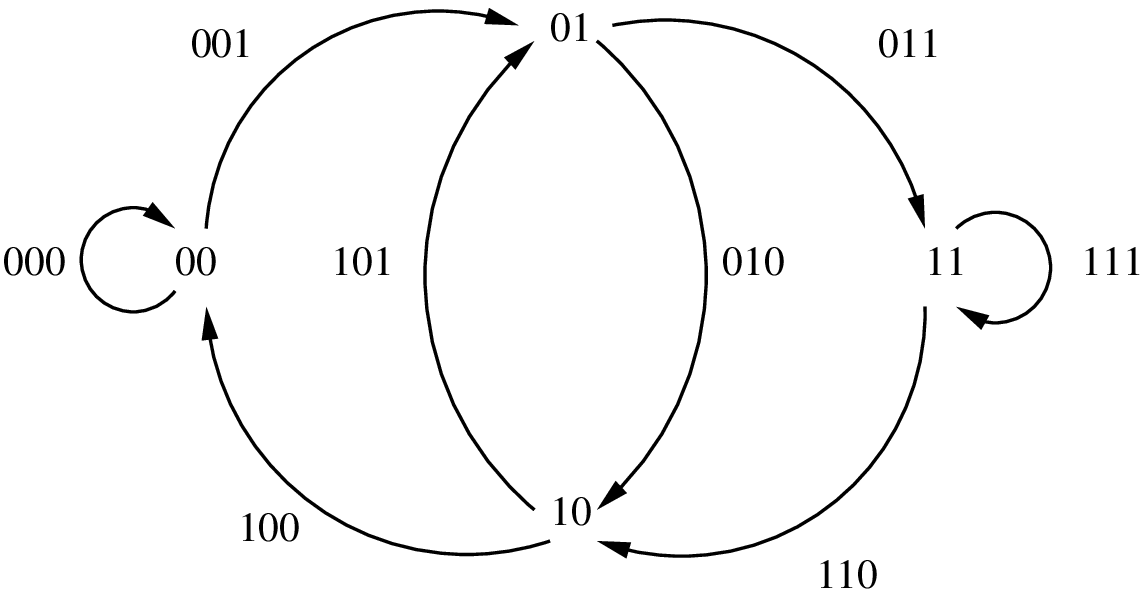}
\end{center}
\caption{The de Bruijn graph $B_2(2)$.}
\end{figure}

\begin{figure}[ht]
\begin{center}
\psfragscanon
\psfrag{000}{$\textbf{000}$}
\psfrag{111}{$\textbf{111}$}
\psfrag{010}{$\textbf{010}$}
\psfrag{100}{$\textbf{100}$}
\psfrag{101}{$\textbf{101}$}
\psfrag{110}{$\textbf{110}$}
\psfrag{001}{$\textbf{001}$}
\psfrag{011}{$\textbf{011}$}
\psfrag{0000}{$0000$}
\psfrag{1111}{$1111$}
\psfrag{0100}{$0100$}
\psfrag{1000}{$1000$}
\psfrag{1010}{$1010$}
\psfrag{1100}{$1100$}
\psfrag{0001}{$0001$}
\psfrag{1110}{$1110$}
\psfrag{0101}{$0101$}
\psfrag{1001}{$1001$}
\psfrag{1011}{$1011$}
\psfrag{1101}{$1101$}
\psfrag{0010}{$0010$}
\psfrag{0111}{$0111$}
\psfrag{0011}{$0011$}
\psfrag{0110}{$0110$}
\includegraphics[scale=.75]{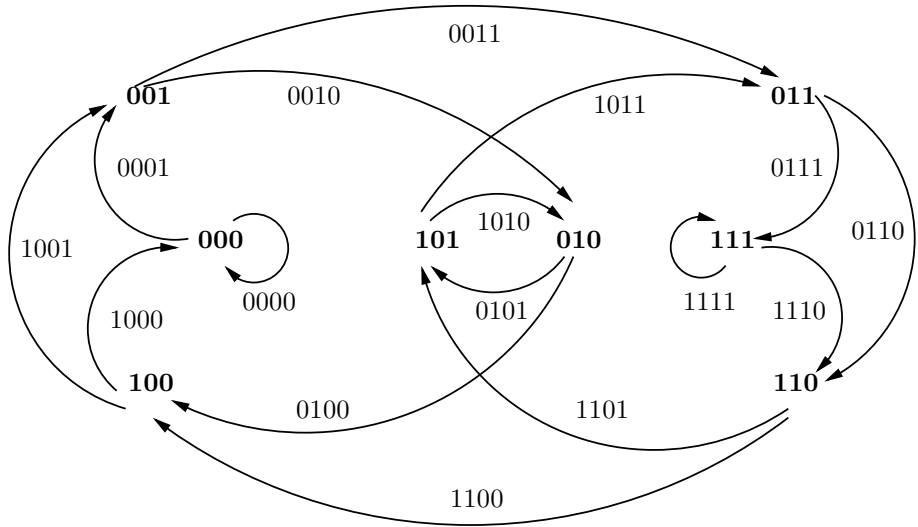}
\end{center}
\caption{The de Bruijn graph $B_2(3)$.}
\end{figure}

\begin{definition}
Let $G=(V,E)$ be a directed graph. The line graph $\mathcal{L}G$ of $G$ is a directed graph such that $V(\mathcal{L}G) = E(G)$, and there is an edge $(e,f)$ for every pair of edges $e,f \in E(G)$ with $t(e)=s(f)$. 
\end{definition}

\begin{figure}[ht]
\begin{center}
\psfragscanon
\psfrag{e1}{$e_1$}
\psfrag{e2}{$e_2$}
\psfrag{e3}{$e_3$}
\psfrag{e4}{$e_4$}
\psfrag{e5}{$e_5$}
\includegraphics[scale=1]{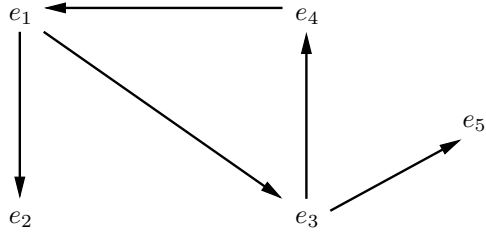}
\caption{The line graph of  the graph in Figure \ref{dir:line}.}
\end{center}
\end{figure}

\begin{definition}
A tour in a directed graph $G$ is a sequence $(v_0, e_1, v_1, e_2, \ldots, e_n, v_n)$ such that $e_i = (v_{i-1}, v_i) \in E$ for each $i = 1, 2, \ldots, n$, and moreover, $e_i \neq e_j$ whenever $i \neq j$.
\end{definition}

\begin{definition}
A directed graph $G$ is Eulerian if there exists a closed directed tour containing all vertices and passing every directed edge exactly once. 
\end{definition}

\begin{definition}
The symmetrization of a directed graph $G=(V,E)$  is the undirected graph $sym(G) = (V, \tilde{E})$, where $$\tilde{E} = \{ \{x,y\} : (x,y) \in E \text{ or } (y,x) \in E \}.$$
\end{definition}

\begin{theorem}
A directed graph $G$ is Eulerian if and only if it's symmetrization is connected, and $deg^+(v)=deg^-(v)$ for every $v \in V(G)$.
\end{theorem}

\begin{theorem}
The de Bruijn graph $B_k(n)$ is Eulerian.
\label{eul:th}
\end{theorem}

\begin{proof}
We already showed that $B_k(n)$ is connected, therefore its symmetrization is connected. Note that $deg^+(u)=deg^-(v)=k$, so $B_k(n)$ is Eulerian.
\end{proof}

\begin{definition}
A Hamiltonian tour in a directed graph $G$ is a tour that visits each vertex of $G$ exactly once. A Hamiltonian cycle is a Hamiltonian tour that is a cycle. If a graph $G$ has a Hamiltonian cycle then we say that $G$ is Hamiltonian.
\end{definition}

\begin{theorem}
The line graph $\mathcal{L}G$ of an Eulerian directed graph $G$ is Hamiltonian.
\end{theorem}

\begin{proof}
Let $C= (v_0, e_1, v_1, e_2, \ldots, e_n, v_0)$ be a closed Eulerian tour in $G$. Note that $e_i = (v_{i-1}, v_i)$. Thus $f_i = (e_i, e_{i+1}) \in E(\mathcal{L}G)$ for $1 \le i \le n-1$, and $ f_n = (e_n, e_1) \in E(\mathcal{L}G)$.  So $(e_1, f_1, e_2, f_2,  \ldots, e_n, f_n, e_1)$ is a Hamiltonian cycle in $\mathcal{L}G$.
\end{proof}

The above theorem holds for both directed and undirected graphs.\\

\begin{theorem}
The de Bruijn graph $B_k(n)$ is Hamiltonian.
\end{theorem}

\begin{proof}
First consider $B_k(1)$. Note that $w=012\ldots (k-1)0$ corresponds to a Hamiltonian cycle in $B_k(1)$, so $B_k(1)$ is Hamiltonian. 
Suppose that $n \ge 2$. Note that $B_k(n)=\mathcal{L}B_k(n-1)$. By Theorem \ref{eul:th} we know that $B_k(n-1)$ is Eulerian, so $B_k(n)$ is Hamiltonian.
\end{proof}

Recall that a word of length $N$ over $\mathcal{A}_k$ with subword complexity $p_w(i)=\min\{k^i, N-i+1\}$ for $1 \le i \le N$ is called a de Bruijn word. 

Fix $n \ge 0$. Let $N=k^n+n-1$. Consider the de Bruijn graph $B_k(n)$. Let $P$ be a path obtained by removing one edge from a Hamiltonian cycle in $B_k(n)$. The length of $P$ is $k^n-1$ and it corresponds to a word $w$ of length $k^n+n-1$ that contains all the $k^n$ words of length $n$ over $\mathcal{A}_k$ as subwords, each exactly once. Thus $p_w(n)=k^n$.

Note that $p_w(i)=k^i$ for $1 \le i \le n$. Also, $p_w(n)=k^n=N-n+1$ and $p_w(i)= N-i+1$ for $n \le i \le N$. Therefore $w$ is a de Bruijn word.
 
Next we will show that any de Bruijn word $u$ of length $N=k^n+n-1$ over $\mathcal{A}_k$ corresponds to a Hamiltonian cycle in $B_k(n)$ with one edge removed.

By the definition of the de Bruijn word, $$p_u(n)=\min\{k^n, N-n+1\}=\min\{k^i, k^n+n-i\}.$$ In particular, $p_u(n)=k^n$. This means that $u$ contains  all the $k^n$ words of length $n$ over $\mathcal{A}_k$ as subwords, each exactly once. Therefore $u$ corresponds to a path in $B_k(n)$ of length $k^n-1$ that contains each vertex of $B_k(n)$ exactly once. Thus $u$ corresponds to a Hamiltonian cycle in $B_k(n)$ with one edge removed.

We showed a bijection between the set of de Bruijn words of length $N=k^n+n-1$ and Hamiltonian cycles in $B_k(n)$ with one edge removed. In particular we proved the following proposition. 

\begin{figure}[ht]
\begin{center}
\psfragscanon
\psfrag{100}{$100$}
\psfrag{110}{$110$}
\psfrag{001}{$001$}
\psfrag{011}{$011$}
\psfrag{000}{$000$}
\psfrag{111}{$111$}
\psfrag{010}{$010$}
\psfrag{101}{$101$}
\includegraphics[scale=.65]{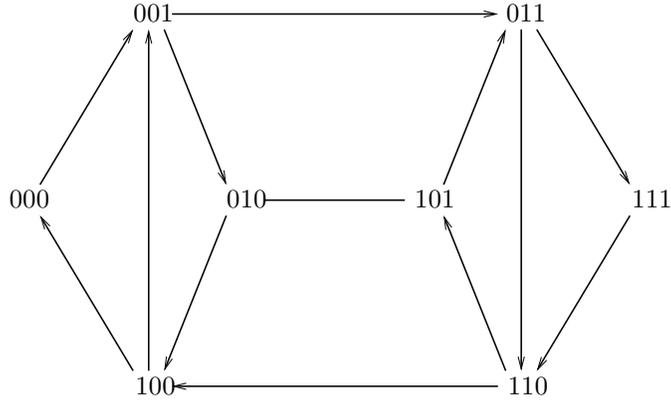}
\caption{The line graph for $B_2(2)$.}
\end{center}
\end{figure}

\begin{proposition}
For every $n > 0$, there exists a de Bruijn word over $\mathcal{A}_k$ of length $N = k^n + n -1$. 
\label{n:pr}
\end{proposition}

Next we generalize Proposition \ref{n:pr}.

\begin{theorem}
For every $N>0$, there exists a de Bruijn word of length $N$ over $\mathcal{A}_k$.
\end{theorem}

\begin{proof}
First consider the case when $N <k$. Let $w=0123 \ldots (N-1)$. Note that $p_w(i)=N-i+1$ for all $i$ such that $1 \le i \le N$. We have that $N <k$ and thus $N-i+1 < k^i$ for all $i >0$. It follows that $\min\{k^i, N-i+1\} = N-i+1$ for all $i>0$ and therefore $w$ is a de Bruijn word.
 
Suppose $N \ge k$. Let $n$ be the largest positive integer such that $ k^n+n-1\le N$. If $k^n+n-1=N$, the claim of the theorem follows from Proposition \ref{n:pr}.

Suppose $k^n+n-1< N$. Let $C = (v_0, e_1, v_1, e_2, \ldots, e_{k^n}, v_0)$ be a Hamiltonian cycle in $B_k(n)$. Now consider the graph $G = B_k(n) - \{e_1, e_2, \ldots, e_{k^n}\} $. Let $K_1, K_2, \ldots, K_l$ be the connected components of $G$ with $m_1, m_2, \ldots, m_l$ edges respectively. Note that $m_1 +m_2+ \ldots +m_l = k^{n+1}-k^n$. Since the indegree and outdegree of each vertex of $G$ is $k-1$, each connected component $K_i$ is Eulerian.
Let $C_1, C_2, \ldots ,C_l$ be Eulerian circuits in $K_1, K_2, \ldots, K_l$ respectively. On each Eulerian circuit $C_i$ choose a vertex $u_i$.  Let $r$ be the minimum positive integer such that $$\sum_{i=1}^r m_i \ge N-k^n-n.$$

We will construct a tour $T$ of length $N-n$ in $B_k(n)$ which contains all the vertices of $B_k(n)$. The tour $T$ will start at vertex $u_r$ and then follows the cycle $C$ until it reaches one of the vertices $u_i$, where $1 \le i \le r-1$. Once it reaches $u_i$, it follows the circuit $C_i$ around and back to $u_i$, and then continues along $C$ until it reaches $u_r$. The tour $T$ then follows the first $N-k^n-n- \sum_{i=1}^{r-1}m_i$ edges of $C_r$. Note that the tour $T$ contains all the vertices of $B_k(n)$ and $N-n$ distinct edges of $B_k(n)$. Let $w$ be the word that corresponds to $T$.

We claim that $w$ is a de Bruijn word. Note that $p_w(i)=k^i$ for $1 \le i \le n$. Also $p_w(n+1)=N-n$ and, in general, $p_w(i)=N-i+1$ for $n<i \le N$. Since $k^n+n-1 <N$, $\min\{k^i, N-i+1\}=k^i$ for $1 \le i \le n$ and $\min\{k^i, N-i+1\}=N-i+1$ for $n<i\le N$. Hence $p_w(i)=\min\{k^i, N-i+1\}$ for $1 \le i \le N$ and, by our definition, $w$ is a de Bruijn word of length $N$ over $\mathcal{A}_k$.

\end{proof}

\begin{figure}[ht]
\begin{center}
\psfragscanon
\psfrag{1}{$1$}
\psfrag{n}{$n$}
\psfrag{p}{$p_w(n)$}
\psfrag{N}{$N$}
\psfrag{L}{$\log_k(N)$}
\includegraphics[scale=.4]{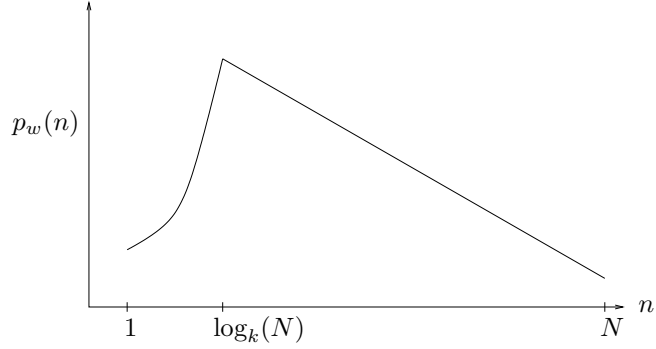}
\caption{The subword complexity function of a de Bruijn word of length $N$ over $\mathcal{A}_k$.} 
\end{center}
\end{figure}

We have just shown that for every $n \ge 0$, there exists a word $w$ of length $N$ over $\mathcal{A}_k$ with subword complexity $p_w(n)= \min\{k^n, N-n+1\}$ for $1 \le n \le N$. Theorem \ref{mineq:th} then follows from this and Theorem \ref{min:th}.


\bigskip
\section{Infinite Sturmian Words}
Before we develop a theory of finite words with low subword complexity, we need to discuss infinite words of low subword complexity. In this section we introduce infinite Sturmian words, which are aperiodic, infinite words of minimal subword complexity. Infinite Sturmian words appear in dynamical systems, and have a nice geometric interpretation. The notation in this section is taken from Lothaire \cite{L:2002}.\\

We denote by $\mathcal{A}^\mathbb{N}$ the set of right-infinite words, and $\mathcal{A}^\infty = \mathcal{A}^* \cup \mathcal{A}^\mathbb{N}$ is the set of all finite or infinite words. A finite word $u$ is a subword of an infinite word $w$ if $w=puq$ for $p,q \in \mathcal{A}^\infty$. 

\begin{definition}
An infinite word $w$ is ultimately periodic if $w=uv^\infty$ for finite words $u,v$ and $v \neq \epsilon$. If $w$ is a finite nonempty word, then $w^\infty = wwww\dots$ is called purely periodic. \\
\end{definition}

\begin{theorem}
Let $w$ be an infinite word.
\begin{enumerate}
\item Then $p_w(n) \le p_w(n+1)$ for $n \ge 1$.
\item If $\exists N$ such that $p_w(N)=p_w(N+1)$, then $w$ is ultimately periodic and $p_w(n)=p_w(N)$ for all $n \ge N$.
\item If $w$ is not ultimately periodic, then $p_w(n) \ge n+1$.
\end{enumerate}
\label{per:th}
\end{theorem}

\begin{proof}
Let $w$ be an infinite word.
\begin{enumerate}
\item Since $w$ is infinite, we will always be able to prolong a subword on the right, and therefore each subword of length $n$ is a prefix for at least one subword of length $n+1$ for all $n \ge 1$. Distinct subwords of length $n$ of $w$ are prefixes of distinct subwords of length $n+1$ of $w$. Thus $p_w(n) \le p_w(n+1)$ for all $n \ge 1$.

\item Let $N$ be an integer such that $p_w(N)=p_w(N+1)$. This means that there are no special subword of length $N$ in $w$. Let $n >N$. Then there are no special subwords of length $n$ in $w$ because otherwise the suffix of length $N$ of such a special subword would be special. We conclude that $p_w(n)=p_w(n-1)=\ldots = p_w(N)$.

Since $w$ is infinite, there exists a subword of length $N$ in $w$ that occurs an infinite number of times in $w$. At least two such occurrences are non-overlapping. Then $w=uvu'vw'$, where $u$ and $u'$ are finite words and $w'$ is an infinite word. Since there are no special subwords of length $\ge N$ in $w$, every occurrence of a word of length $\ge N$ in $w$ is followed by the same letter in $w$. Thus $w'$ has prefix $u'v$, and moreover, each occurrence of $u'v$ is followed by $u'v$. Hence $w=u(u'v)^\infty$ and $w$ is ultimately periodic.

\item By part (2) it follows that if $w$ is not ultimately periodic, then $p_w(n) < p_w(n+1)$ for all $n \ge 1$. Thus $p_w(1) \ge 2$ and $p_w(n+1) \ge p_w(n)+1$. It follows by induction that $p_w(n) \ge n+1$ for all $n \ge 1$.
\end{enumerate}
\end{proof}

\begin{proposition}
Let $w$ be a right-infinite word. The following statements are equivalent:
\begin{enumerate}
\item $w$ is ultimately periodic,
\item there exists $N$ such that $p_w(N)=p_w(N+1)$,
\item there exists $N$ such that, for all $n \ge N$, $p_w(n) = p_w(N)$,
\item there exists $C$ such that $p_w(n) < C$ for all $n$.
\end{enumerate}
\end{proposition}

\begin{proof}
The proof follows from Theorem \ref{per:th}.
\end{proof}

\begin{definition}
A Sturmian word $w$ is a right-infinite word satisfying $p_w(n)=n+1$ for all integers $n \ge 0$. 
\label{51:def}
\end{definition}

Note that since $p_w(1)=2$, infinite Sturmian words always contain exactly two distinct letters. 

The following definition of a Sturmian word is equivalent to Definition \ref{51:def}.

\begin{definition}
An infinite word is Sturmian if it is binary and there is exactly one special subword of each length.
\end{definition}

It follows from Theorem \ref{per:th} that an infinite Sturmian word has the lowest subword complexity among aperiodic infinite words. 
\begin{definition}
The height of a finite word $w$ over $\mathcal{A}_2 = \{0,1\}$ is the number $h(w)$ of occurrences of the letter 1 in $w$. Given two words $u$ and $v$ of the same length, their balance $\delta(u,v)$ is the number
\[ \delta(u,v) =  \left| h(u) - h(v) \right| .\]
\end{definition}

\begin{definition}
We say a set of words $X$ is balanced if for $u,v \in X$, 
\[\left|u\right| = \left|v\right| \Rightarrow \delta(u,v) \le 1.\]
A finite or infinite word $w$ is balanced if its set of subwords $X=Sub(w)$ is balanced.\\
\end{definition}

The statements of Propositions \ref{unbal:pr}, \ref{bal:pr}, \ref{rec:pr}, \ref{pref} and Theorem \ref{SBA:th} are taken from Lothaire \cite{L:2002} and Fogg \cite{PF:2002}, and the proofs of the statements have been modified for this paper.
 
\begin{proposition}
If $w$ is an unbalanced infinite word, then there exists a subword $w'$ of $w$ such that $0w'0$ and $1w'1$ are both subwords of $w$.
\label{unbal:pr}
\end{proposition}

\begin{proof}
Let $w$ be an unbalanced infinite word. Then there exist two subwords $u$ and $v$ of $w$ such that $\left|u\right| = \left|v\right|$ and $\delta(u,v) \ge 2$. Let $m$ be the first position in which $u$ and $v$ differ. Let $u'$ and $v'$ be the suffixes of length $\left|u\right|-m+1$ of $u$ and $v$ respectively. Note that $\left|u'\right|=\left|v'\right|$, $\delta(u',v') \ge 2$ and the first character of $u'$ and $v'$ are distinct. By replacing $u,v$ with $u',v'$ we can assume, without loss of generality, that the first characters of $u$ and $v$ are distinct. 
\par Let $u=u_1u_2\ldots u_n$ and $v=v_1v_2 \ldots v_n$ where $u_i, v_i \in \mathcal{A}_k$. Then note that $\delta(u_1,v_1)=1$. Now since $\delta(u,v) \ge 2$, there exists a minimal $j$ such that $$\delta(u_1u_2\ldots u_j, v_1v_2 \ldots v_j)=2.$$  Then $u_1=u_j \neq v_1=v_j$, and $u_i=v_i$ for $2 \le i \le j-1$. Let $w' = u_2 u_3 \ldots u_{j-1}$. Then either $u_1u_2\ldots u_j=0w'0$ and $ v_1v_2 \ldots v_j=1w'1$ or  $u_1u_2\ldots u_j=1w'1$ and $ v_1v_2 \ldots v_j=0w'0$. Since $u_1u_2\ldots u_j$ and $ v_1v_2 \ldots v_j$ are subwords of $w$, the claim follows.\\
\end{proof}

\begin{proposition}
\label{bal:pr}
Let $w$ be an infinite word and let $Sub(w)$ be the set of all subwords of $w$. If $Sub(w)$ is balanced, then for all $n > 0$, $$p_w(n) \le n+1.$$
\end{proposition}

\begin{proof}
The result is obvious for $n=1$. If $n=2$ the claim holds because both $00$ and $11$ cannot be elements of $Sub(w)$. Assume for contradiction that $n \ge 3$ is the smallest integer for which the statement is false. Then $p_w(n-1) \le n$ and $p_w(n) \ge n+2$. For each $v \in Sub_w(n)$, its suffix of length $n-1$ is in $Sub_w(n-1)$. So there exist two distinct words $u, u' \in Sub_w(n-1)$ such that $0u, 1u, 0u', 1u' \in Sub_w(n)$. Since $u \neq u'$, there exists a word $z$ such that $z0$ and $z1$ are prefixes of $u$ and $u'$. But then $0z0$ and $1z1$ are words in $Sub(w)$, showing that $Sub(w)$ is unbalanced.\\
\end{proof}

\begin{proposition}
An infinite Sturmian word $w$ is recurrent, that is, every subword that occurs in $w$ occurs an infinite number of times.
\label{rec:pr}
\end{proposition}

\begin{proof}
Let $w$ be an infinite Sturmian word. Assume for contradiction that a subword $u$ of length $n$ occurs in $w$ only a finite number of times, and assume that $u$ does not occur after the $N$th letter of $w$. Now let $v$ be the right-infinite word obtained by removing the prefix of length $N$ from $w$. Then $v$ is contained in $w$ but does not contain $u$ as a subword, and so $p_v(n) \le n$, but, by Theorem \ref{per:th}, this implies that $v$ is ultimately periodic, a contradiction.\\
\end{proof}

\begin{lemma}
Let $w$ be an infinite word, let $n \ge 1$, and let $c $ be the number of subwords of length $n$ and valence 1 in $w$. If $w$ has a subword $u$ of length $n+c$ whose subwords of length $n$ are all of valence 1, then $w$ is eventually periodic. 
\label{conserv:lm}
\end{lemma}

\begin{proof}
Let $w$ be an infinite word, $n \ge 1$, and let $c$ be the number of subwords of length $n$ and valence 1 in $w$. Let $u=u_1u_2 \ldots u_{n+c}$ be a subword of length $n+c$ of $w$ such that all subwords of $u$ of length $n$ have valence 1. We want to show that $w$ is ultimately periodic. Notice that there are $(n+c)-n+1=c+1$ not necessarily distinct subwords of length $n$ in $u$. Since all subwords of length $n$ in $u$ are of valence 1 and there are $c$ distinct subwords of length $n$ and valence $1$ in $w$, a subword of length $n$, say $v$,  occurs at least twice in $u$. If we consider the second occurrence of $v$ in $u$, we already know the letters that follow $v$ by looking at the letter that follow the first occurrence of $v$ in $u$. We also know that none of these longer subwords have higher valence, because the suffix of each such word of length $>n$ is a subword of length $n$ in $u$, and therefore also has valence $1$. We can continue adding letters onto $v$ until we reach $v$ again, and this process continues indefinitely. Thus $w$ is periodic.\\
\end{proof}

\begin{theorem}
Let $w$ be an infinite word. The following conditions are equivalent:
\begin{enumerate}
\item $w$ is Sturmian,
\label{it1}
\item $w$ is balanced and aperiodic.\\
\label{it2}
\end{enumerate}
\label{SBA:th}
\end{theorem}

\begin{proof}
$(\ref{it2}) \Rightarrow (\ref{it1}):$ Let $w$ be balanced and aperiodic. Then by Proposition \ref{bal:pr} and by Theorem \ref{per:th}, we have that for all $n \ge 1$, $p_w(n) \le n+1$ and  $p_w \ge n+1$. Thus $p_w(n)=n+1$ for all $n \ge 1$ and so $w$ is infinite Sturmian.\\
\par $(\ref{it1}) \Rightarrow (\ref{it2}):$ Let $w$ be infinite Sturmian. Then $p_w(n)=n+1$ for all $n \ge 1$, so $p_w(n) \neq p_w(n+1)$ for any $n$. Thus $w$ is aperiodic by Theorem \ref{per:th}.
\par By contradiction suppose that $w$ is not balanced. We will show that $w$ is ultimately periodic. By Proposition \ref{unbal:pr} there must exists a subword $v$ of $w$ such that $0v0$ and $1v1$ are both subwords of $w$. Consider such a word $v=v_1v_2 \ldots v_{n}$, $v_i \in \mathcal{A}_2$, of minimal length $n$. Note that $v \neq \epsilon$ because otherwise $00$ and $11$ would be subwords of length two in $w$, by Proposition \ref{rec:pr} that 0 and 1 occur an infinite number of times in $w$, so 01 and 10 must occur, which would imply that $p_w(2)=4$. But $w$ is Sturmian, so $p_w(2)=3$.

We will prove that $v$ is a palindrome, that is, $v_i=v_{n-i+1}$ for $1 \le i \le n$. Assume $v$ is not a palindrome. Let $j \ge 1$ be the first index such that $v_j \neq v_{n-j+1}$. Without loss of generality let $v_j=0$ and $v_{n-j+1}=1$. Then we have that $0v_1 \ldots v_{j-1}0$ and $1v_{n-j+2} \ldots v_n1$ is an unbalanced pair in $w$ of shorter length, contradicting the minimality of $v$. 

Since $w$ is Sturmian, we know there are $n+1$ distinct subwords of $w$ of length $n$. Note that  $v$ is a special subword in $w$, and therefore is a suffix of a special subword of length $n+1$. There is exactly one special subword of length $n+1$. Suppose that $0v$ is special and thus $1v$ is not, therefore $0v1$ is a subword of $w$ and $1v0$ is not. 

Let $i$ be the index of an occurrence of $1v1$ in $w$. We claim that the subword $0v$ cannot occur in $u=w_i w_{i+1} \ldots w_{i+2n+1}$. The length of $u$ is $2n+2$. The length of $1v1$ is $n+2$ and the length of $1v$ is $n+1$. Suppose that a prefix of $0v$ equals a suffix of $1v1$, then there exists $k$ such that  $0v_1 \ldots v_{n-k+1} = v_k v_{k+1} \ldots v_n1$. But this implies that $v_k=0$ and $v_{n-k+1}=1$, a contradiction to $v$ being a palindrome. It follows that $0v$ is not a subword of $u=w_i w_{i+1} \ldots w_{i+2n+1}$.

There are exactly $n+2$ not necessarily distinct subwords of length $n+1$ in $u$. Since $w$ is Sturmian, there are $n+2$ distinct subwords of length $n+1$ in $w$, one of them is $0v$.  One of the subwords of length $n+1$ of $u$ must occurs at least twice, because $u$ is a subword of $w$ and $0v$ does not occur in $u$. Since $0v$ is the only special subword of length $n+1$ in $w$, all the subwords of length $n+1$ of $u=w_i w_{i+1} \ldots w_{i+2n+1}$ are not special, that is, have valence 1. Thus, by Lemma \ref{conserv:lm}, $w$ is ultimately periodic, a contradiction. It follows that $w$ is balanced.

\end{proof}

\begin{definition}
A function $\varphi : \mathcal{A}^* \rightarrow \mathcal{B}^*$ is called a morphism (or substitution) if $\varphi(xy) = \varphi(x)\varphi(y)$ for every $x,y \in \mathcal{A}^*$. We say a morphism $\varphi$ is nonerasing if the image of every letter is a nonempty word.
\end{definition}

\begin{definition}
We say that a word $x$ is a fixed point of a morphism $\varphi$ if $x=\varphi(x)$.
\end{definition}
 
 \begin{proposition}
 Let $\varphi$ be a nonerasing morphism from $\mathcal{A}^*$ to itself, and let $a$ be a letter such that $\varphi(a)=ab$ for some nonempty word $b$. For $n \ge 0$, set \[ u_n=\varphi^n(a), \,\,\,\,\, v_n=\varphi^n(b).\] Then
 \begin{enumerate}
 \item $u_{n+1}=u_nv_n$, so $u_n$ is a prefix of $u_{n+1}$ for all $n \ge 0$,
 \item $u_{n+1}=av_0v_1v_2 \ldots v_n$,
 \item the infinite word $$w=ab\varphi(b)\varphi^2(b)\dots \varphi^n(b) \dots$$ is the direct limit of the sequence of words $u_n$ as $n \to \infty$. We write $w = \lim_{n \to \infty}u_n$. The word $w$ is the unique fixed point of $\varphi$ starting with the letter $a$. We call $w$ a morphic word.
 \end{enumerate}
 \label{pref}
 \end{proposition}
 
 \begin{proof}
 \begin{enumerate}
 \item We have $$u_{n+1}=\varphi^{n+1}(a)=\varphi^n(\varphi(a))=\varphi^n(ab)=\varphi^n(a)\varphi^n(b)=u_nv_n.$$ Thus $u_n$ is a prefix of $u_{n+1}$ for all $n \ge 0$.
 \label{part1}
 \item By part (\ref{part1}) we have $u_1= av_0$. Proceeding by induction, assume $u_{n}=av_0v_1v_2 \ldots v_{n-1}$ for some $n$. Then \[
 u_{n+1} = u_n v_n= av_0 v_1\dots v_{n-1}v_n.
 \]
 \item It is clear that $w=ab\,\varphi(b)\varphi^2(b) \dots \varphi^n(b) \dots$ is the direct limit of $u_n$ as $n \to \infty$. Since $$\varphi(w)=\varphi(a) \varphi(b) \varphi^2(b) \dots = ab\varphi(b)\varphi^2(b)\dots = w,$$  $w$ is a fixed point of $\varphi$.
 
 Now assume that $x$ is another fixed point of $\varphi$ starting with $a$. We will show by induction that $\forall n$, $ab\varphi(b)\varphi^2(b)\varphi^3(b) \dots \varphi^n(b)$ is a prefix of $x$, and, therefore $x=w$.
 
 The claim holds for $n=1$. Note that $x$ starts with letter $a$ and $\varphi(a) = ab$. Since $x$ is a fixed point, $\varphi(ab)=\varphi(a)\varphi(b)=ab\varphi(b)$ is a prefix of $x$.
 
 Now suppose that $ab\varphi(b)\varphi^2(b)\varphi^3(b) \dots \varphi^n(b)$ is a prefix of $x$. We want to show that $ab\varphi(b)\varphi^2(b)\varphi^3(b) \dots \varphi^{n+1}(b)$ is a prefix of $x$. Since  $\varphi(x)=x$ we get that
 
 \begin{align*}
 \varphi(ab\varphi(b)\varphi^2(b) \dots \varphi^n(b)) & = \varphi(a)\varphi(b)\varphi^2(b)\varphi^3(b) \dots \varphi^{n+1}(b)\\
  &= ab\,\varphi(b)\varphi^2(b)\varphi^3(b) \dots \varphi^{n+1}(b)
  \end{align*}
 is a prefix of $x$.
 
 \end{enumerate}
 \end{proof}

\begin{example}
Let $\mathcal{A}=\{0,1\}$. Consider the nonerasing morphism $\varphi$ defined by
\begin{align*}
&\varphi(0)=01,\\
&\varphi(1)= 0.
\end{align*}

Consider the word $f_0=0$ and define $f_n = \varphi(f_{n-1})$. Note that $\varphi(f_0) =01$. Using the notation of Proposition \ref{pref} we have that 
\[ u_n = \varphi^n(0) = f_n, \,\,\,\,\, v_n = \varphi^n(1)=f_{n-1}\] for $n \ge 1$. By Proposition \ref{pref}, $f_n$ is a prefix of $f_{n+1}$ for all $n \ge 0$. We also have that $\lim_{n \to \infty} f_n$ exists, and then the infinite morphic word $f$ defined by
\[
f = \lim_{n \to \infty} f_n = 01001010010010100101001001\ldots
\]
is the unique fixed point of $\varphi$ starting with the letter 0. We also have that
\begin{equation}
f_{n+1}=f_n f_{n-1}.
\label{fibrec}
\end{equation}
The word $f= \lim_{n \to \infty} f_n = 01001010010010100101001001\ldots$ is called the \textit{Fibonacci word}. 

Note that the sequence of the lengths of the words $f_n$ is the traditional Fibonacci sequence. Equation (\ref{fibrec}) gives a recursive definition of the sequence $f_n$.\\
\label{fib:ex}
\end{example}

\begin{proposition}
The Fibonacci word $f$ is infinite Sturmian. 
\label{fibinSt}
\end{proposition}

\begin{proof} To prove that $f$ is infinite Sturmian, we need to show there there is exactly one special subword of each length. Note that $f$ is a concatenation of $01$s and $0$s. Thus $11$ is not a subword of $f$, and so $p_f(2)=3$. 

We first show that $f$ is balanced. By Proposition \ref{bal:pr} we need to show that, for any word $u$, both $0u0$ and $1u1$ cannot be subwords of $f$. This claim will be proved using induction on the length of $u$. If $u$ is the empty word, then we have just established that $1u1=11$ is not a subword of $f$. Assume for contradiction that there exists a subword $u$ of minimal length such that both $0u0$ and $1u1$ are subwords of $f$. Note that $u$ must start and end with $0$, otherwise there exists a subword $v$ of $f$ such that $11v11$ is a subword of $f$, which would imply that $11$ is a subword of $f$. So $u=0v0$ for some subword $v$ of $f$. Then $00v00$ and $10v01$ are subwords of $\varphi(f)$. Since $f = \varphi(f)$, there exists a subword $z$ of $f$ such that $\varphi(z)=0v$. Then, by the definition of $\varphi$, $00v0=\varphi(1z1)$ and $010v01 = \varphi(0z0)$, this implies that $0z0$ and $1z1$ are both subwords of $f$. But $\left|z\right| < \left|u\right|$, which contradicts the minimality of $u$. It follows that $f$ is balanced.

To show that $f$ has at most one special subword of each length, assume for contradiction that both $u$ and $v$ are special subwords of the same length, $u \neq v$, and let $z$ be their longest common suffix. Since $u, v$ are special, $u0, u1, v0, v1$ are all subwords of $f$. Since $z$ is the longest common suffix of $u$ and $v$, we have that $u$ and $v$ differ in the letter preceding $z$. But then $0z0$, $0z1$, $1z0$ and $1z1$ are subwords of $f$. This contradicts the fact that $f$ is balanced.

We proved that $f$ has at most one special subword of each length. We will now prove that $f$ has at least one special subword of each length. 
Recall that for a word $w=w_1 w_2 \ldots w_n$, the reverse word is $w^{-1}=w_n w_{n-1} \ldots w_1$. Define the following words

\[
g_2 = \epsilon \, \, \, \, \mbox{and} \, \, \, \, g_n = f_{n-3} \cdots f_1 f_0 \, \, \, \, \mbox{for}\, \, \, \,  n\ge 3
\]
and
\[
t_n=
 \left\{
	\begin{array}{ll}
		01,  & \mbox{if} \text{ $n$ is odd,} \\
		10,  & \mbox{if} \text{ $n$ is even.}
	\end{array}
\right. \\
\]

We claim that
\begin{equation}
f_{n+2} = g_n \,f^{-1}_n \,f^{-1}_n\, t_n, \, \, \, \, \, \, \,  n \ge 2.
\label{fibclaim}
\end{equation}

This claim will be proved using induction. The relation holds for
\begin{align*}
f_4 & = g_2 \,f_2^{-1}\, f^{-1}_2\, t_2 = \epsilon (010)(010)10=01001010,\\ 
f_5 & = g_3 \,f_3^{-1} \,f_3^{-1} \,t_3 = 0(10010)(10010)01= 0100101001001.
\end{align*}

To prove claim (\ref{fibclaim}), we will use the following three properties of $\varphi$.
\begin{enumerate}
\item[(a)] $\varphi(w^{-1})0 = 0 \varphi(w)^{-1}$ for any word $w$. We prove this by using induction on the length of $w$. The claim holds when $w = \epsilon$ and $w= 0, 1$. Suppose  $\varphi(w^{-1})0 = 0 \varphi(w)^{-1}$ for all words $w$ such that $\left|w\right| = n$. Consider $w'$ such that $\left|w'\right| = n+1$. Consider two cases
\begin{enumerate}
\item[(a1)] If $w' = w0$, then  
\begin{align*}
\varphi(w'^{-1})0 & =  \varphi((w0)^{-1})0  = \varphi(0w^{-1})0\\ & = \varphi(0)\varphi(w^{-1})0  = 01 0\varphi(w)^{-1}\\  & = 0 \varphi(0)^{-1} \varphi(w)^{-1} = 0(\varphi(w)\varphi(0))^{-1}\\ & = 0 \varphi(w0)^{-1}. 
\end{align*}
\item[(a2)] The proof is similar for $w'=w1$.\\
\end{enumerate}
\label{1}

\item[(b)] $\varphi(f_n^{-1} t_n) = 0 f_{n+1}^{-1} t_{n+1}$ . Consider two cases
\begin{enumerate}
\item[(b1)] Let $n$ be even. Then $t_n=10$ and
\[
\varphi(f_n^{-1}t_n) = \varphi(f_n^{-1}10) = \varphi(f_n^{-1})001.
\]
It follows from property (a) that
\begin{align*}
 \varphi(f_n^{-1})001= 0\varphi(f_n)^{-1} 01 = 0 f_{n+1}^{-1} t_{n+1}.
\end{align*}
\item[(b2)] The proof is similar for $n$ odd.\\
\end{enumerate}

\item[(c)] $\varphi(g_n)0 = g_{n+1}$. By the definition of $g_n$ 
\begin{align*}
 \varphi(g_n)0 & = \varphi(f_{n-3} \dots f_0)0 = \varphi(f_{n-3}) \dots \varphi(f_0) 0\\ &= f_{n-2} \dots f_1 0 =  f_{n-2} \dots f_1 f_0\\ &= g_{n+1}.\\
 \end{align*} 

\end{enumerate}

Combining properties (1), (2), and (3) we get that 
\begin{align*}
f_{n+3} &= \varphi(f_{n+2}) = \varphi(g_n)\varphi(f_n^{-1})\varphi(f_n^{-1}t_n)\\ 
              &= \varphi(g_n)\varphi(f_n^{-1})0f_{n+1}^{-1}t_{n+1}\\ & =\varphi(g_n)0\varphi(f_n)^{-1}f_{n+1}^{-1}t_{n+1}\\ 
              &= g_{n+1}f_{n+1}^{-1}f_{n+1}^{-1}t_{n+1}.
\end{align*}
This proves claim (\ref{fibclaim}).

Consider $f_{n+2} = g_n \,f^{-1}_n \,f^{-1}_n\, t_n$. Observe that the first letter of $f_n^{-1}$ is the opposite of the first letter of $t_n$. This is because the last letter of $f_n$, which is the first letter of $f_n^{-1}$, is 0 when $n$ is even and 1 when $n$ is odd, which is proved by induction using the recursive definition of the Fibonacci word. Thus $f_n$ is a special subword of $f_{n+2}$ for $n \ge 2$, and, therefore, a special subword of $f$. Since the suffix of a special subword is also a special subword, this proves that special subwords of any length exist.

We've shown that there exists at most one special subword of each length and at least one special subword of each length, and thus there exists exactly one special subword of $f$ of every length. Therefore $f$ is an infinite Sturmian word.

\end{proof}

\textbf{Geometric Interpretation.} There is a well known geometric interpretation of Sturmian words that involves relating Sturmian words to lines. Let $\theta > 0$ be an irrational real number, and consider the line $L_\theta$ given by $y = \theta x$. Following this line to the right, we can define a word $s$ by

\[
s_i = 
 \left\{
	\begin{array}{ll}
		0  & \mbox{if } \text{$L$ crosses a vertical grid-line} \\
		1  & \mbox{if } \text{$L$ crosses a horizontal grid-line}
	\end{array}
\right.
 \]

\begin{figure}[ht]
\begin{center}
\psfragscanon
\psfrag{0}{$1$}
\psfrag{1}{$0$}
\psfrag{0 1}{$1\, 0$}
\psfrag{t}{$\theta$}
\psfrag{y=tx}{$y=\theta x$}
\includegraphics[scale=.5]{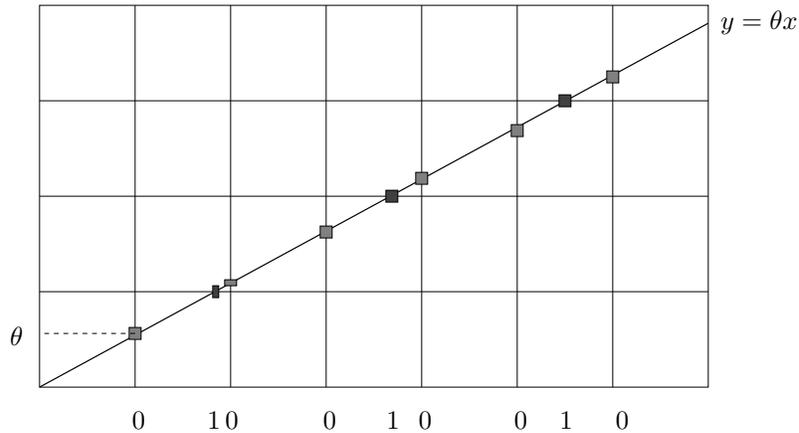}
\caption{The cutting sequence given by line $y=\theta x$.}
\label{cutseq:ex}
\end{center}
\end{figure}

so the resulting infinite word would be
\[ s_\theta = s_1s_2s_3 \ldots \]
This $s_\theta$ is sometimes called a \textit{cutting sequence}, and is a Sturmian word. Figure \ref{cutseq:ex} shows a graph of such a line $L_\theta$. 

If $\theta$ is rational, this construction corresponds to a periodic word. Note that not only do all such lines $y=\theta x$, for $\theta$ irrational, correspond to Sturmian words, but given a Sturmian word $w$, we can find a cutting-sequence representation of $w$. This interpretation allows for nice results and is often used as a tool to prove properties of Sturmian words. 

Using this cutting sequence interpretation, we get another equivalent definition of a Sturmian word. \textit{Mechanical words}, or \textit{rotation words}, are the infinite words defined for $0 < \alpha < 1$ and $0 \le \rho \le 1$ by
\[
s_{\alpha,\rho}(n) = 
 \left\{
	\begin{array}{ll}
		0,  & \mbox{if $\lfloor (n+1)\alpha + \rho \rfloor= \lfloor n\alpha + \rho \rfloor$,}\\
		1,  & \mbox{otherwise,}\\
	\end{array}
\right.
\]
and
\[
s'_{\alpha,\rho}(n) = 
 \left\{
	\begin{array}{ll}
		0,  & \mbox{if $\lceil (n+1)\alpha + \rho \rceil= \lceil n\alpha + \rho \rceil$,}\\
		1,  & \mbox{otherwise.}
	\end{array}
\right.
 \]
for $n \ge 0$. The word $s_{\alpha,\rho}$ ( $s'_{\alpha,\rho}$) is called the lower (upper) mechanical word with slope $\alpha$ and intercept $\rho$. When $\alpha$ is irrational, mechanical words are Sturmian. Also, any Sturmian word is a mechanical word.
 

Due to these geometric interpretations, Sturmian words are now receiving some attention in computer graphics and image processing. For example, counting the number of essentially different digitized straight lines corresponds to counting the number of subwords of length $n$ in all Sturmian words, that is, the number of all finite Sturmian words of length $n$ (see Section 6). De Luca and Mignosi give a proof of this is \cite{LM:1994}. For more on mechanical words, see Berstel \cite{JB:2007}, or Lothaire \cite{L:2002}.

\bigskip
\section{Finite Sturmian Words}
In this section we will consider finite words of low subword complexity, in particular we will talk about finite Sturmian words.

\begin{definition}
A finite word $s$ is called Sturmian if $s$ is a subword of an infinite Sturmian word.
\label{finstu:def}
\end{definition}

\begin{proposition}
Let $w$ be an infinite Sturmian word. The word that results from removing a finite prefix from $w$ is infinite Sturmian.
\label{prefSt}
\end{proposition}

\begin{proof}
Let $w$ be an infinite Sturmian word, and let $p$ be a finite prefix of $w$. Let $w'$ be the infinite word obtained by removing prefix $p$ from $w$. We want to show that $w'$ is infinite Sturmian. We need to show that $p_{w'}(n)=n+1$ for all $n$. Recall from Proposition \ref{rec:pr} that any subword $u$ of $w$ occurs infinitely many times in $w$, and thus every subword of $w$ occurs in $w'$. So $p_{w'}(n)=p_{w}(n)=n+1$ for all $n$.\\ 
\end{proof}

The following equivalent definition of finite Sturmian words follows from Proposition \ref{prefSt}.

\begin{definition}
A finite word $s$ is Sturmian if it is the prefix of an infinite Sturmian word.
\label{prefSt}
\end{definition}

The following proposition is from de Luca \cite{AL:1999} and modified for this paper.
\begin{proposition}
Let $w$ be a finite Sturmian word of length $N$. Then
\[ N = R_w + K_w \]
where $R_w$ and $K_w$ are defined in Section 2.
\end{proposition}

\begin{proof}
Let $m = \min\{R_w,K_w\}$ and $M= \max\{R_w,K_w\}$. For $n \in [1,m]$ we know by Theorem \ref{main} that subword complexity sequences are strictly increasing, so $p_w(n) \ge n+1$. But $w$ is a subword of an infinite Sturmian word, so we have that $p_w(n) \le n+1$ and thus $p_w(n)=n+1$. Now to show that $p_w(n+1) = p_w(n)$ for $n \in [m,M]$ we consider two cases:
\begin{enumerate}
\item $R_w < K_w$\\
It follows from Theorem \ref{main} that $p_w$ is constant on $[R_w,K_w]$ and $p_w(R_w)=p_w(K_w) = R_w + 1$.
\item $K_w \le R_w$\\
By Theorem \ref{main}, $p_w$ is nondecreasing on $[K_w, R_w]$. When $n < R_w$, $w$ contains at least special subword of length $n$. Since $w$ is Sturmian, there is at most one special subword of length $n$ in $w$. Therefore $w$ contains exactly one special subword of length $n$ in $w$. By Equation (\ref{eq2}) in Theorem \ref{main}, $$p_w(n+1)=p_w(n) + s_w(n,2)-1 = p_w(n)+1-1=p_w(n).$$ Thus $p_w$ is constant on $[K_w,R_w]$ and $p_w(K_w)=p_s(R_w)=K_w+1$.
\end{enumerate}
Now by Theorem \ref{main}, $p_w$ is strictly decreasing on the interval $[M,N]$, and for $n \in [M,N]$, $p_w(n+1) = p_w(n)-1$. It follows that
\begin{align*}
&&p_w(m) - (N-M) & = 1 \\
& \Rightarrow &m+1 - (N-M) & =1 \\
& \Rightarrow &m+ M & =N \\
& \Rightarrow &R_w + K_w &= N\\
\end{align*}

\end{proof}

It is important to note that the condition $N = R_w + K_w$ does not characterize finite Sturmian words. Consider the following example:

\begin{example}
Let $w=0011$. Then $\left|w\right| = 4$ and $R_w = 2$ and $K_w = 2$, but $w$ is not a finite Sturmian word since $w$ is not balanced.
\end{example}

\begin{figure}[ht]
\begin{center}
\psfragscanon
\psfrag{1}{$1$}
\psfrag{N}{$N$}
\psfrag{p}{$p_w(n)$}
\psfrag{n}{$n$}
\psfrag{k}{$\frac{N}{2}$}
\includegraphics[scale=.4]{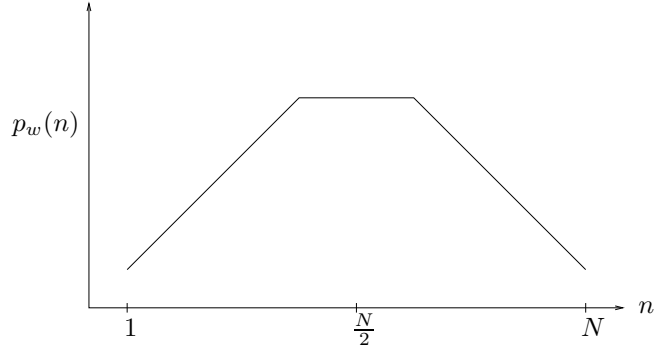}
\caption{The subword complexity function of a finite Sturmian word.} 
\end{center}
\end{figure}

\begin{definition}
A very low complexity word is a finite word $w$ for which there exist positive integers $a,b$, $a <b$, such that $p_w(n)=n+1$ for $1 \le n \le a$, $p_w(n+1)=p_w(n)$ for $a \le n \le b$, and $p_w(n)=N-n+1$ for $b \le n \le N$. 
\end{definition}

Note that if $w$ is a very low complexity word, then $a=\min \{R_w,K_W\}$ and $b= \max \{R_w,K_w\}$. All finite Sturmian words are very low complexity words, however not all very low complexity words are finite Sturmian words. For example $w=0011$ is a very low complexity word, but it is not a finite Sturmian word because it is unbalanced.

It's particularly interesting to find a low complexity word whose subword complexity function does not plateau, or has a plateau of length 1. Such a word is considered in Proposition \ref{peak}. The subword complexity sequence of the word constructed in Proposition \ref{peak} obtains its maximum value at the latest possible length. It follows that only the first $\frac{N}{2}$ values of a subword complexity sequence are significant, and the rest can be extrapolated.

\begin{proposition}
For any $N$, there exists a binary word of length $N$ such that $p_w(n)=n+1$ for $1 \le n \le \lfloor\frac{N}{2}\rfloor$, and $p_w(n)=N-n+1$ for $\lceil \frac{N}{2}\rceil \le n \le N$.
\label{peak}
\end{proposition}

\begin{proof}
We consider two cases:
\begin{enumerate}
\item Let N be even. Consider the word  

\[
w_N =  \underbrace{00\ldots0}_{ \frac{N}{2}-1 \text{ zeroes}} 01 \underbrace{00\ldots0}_{ \frac{N}{2}-1 \text{ zeroes}}.\\
\]

Let $n \le \frac{N}{2}$. Then the distinct subwords of length $n$ are the subword consisting of all 0s, and the subwords that contain a 1. Since the 1 can be in any of the $n$ places, there are $n$ such subwords that contain a 1. So $p_w(n)=n+1$ for all $n \le \frac{N}{2}$.

Now we want to show that $K_w=R_w=\frac{N}{2}$. $K_w=\frac{N}{2}$ because ``100...0" ($\frac{N}{2}-1$ zeroes) is the shortest suffix of multiplicity 1 in $w_N$.  The minimum length $n$ such that all subwords of length $n$ have valence 1 is $n=\frac{N}{2}$. So $R_w = K_w = \frac{N}{2}$.

It follows from Theorem \ref{main} that $p_w$ is decreasing by one on the interval $[\frac{N}{2},N]$, so $p_w(n)=N-n+1$ for all $\frac{N}{2} \le n \le N$. 

\begin{figure}[ht]
\begin{center}
\psfragscanon
\psfrag{1}{$1$}
\psfrag{N}{$N$}
\psfrag{p}{$p_w(n)$}
\psfrag{n}{$n$}
\psfrag{k}{$\frac{N}{2}$}
\includegraphics[scale=.4]{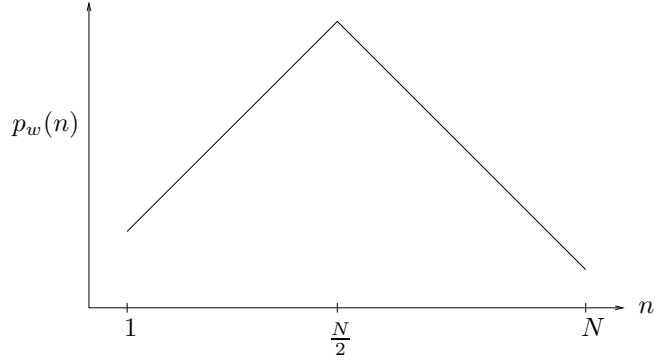}
\caption{Subword complexity sequence for $N$ even.} 
\end{center}
\end{figure}

\item Let $N$ be odd. Consider the word 

\[
w_N =  \underbrace{00\ldots0}_{ \lfloor \frac{N}{2}\rfloor \text{ zeroes}} 1 \underbrace{00\ldots0}_{ \lfloor\frac{N}{2}\rfloor \text{ zeroes}}.\\
\]

Let $n \le \lfloor \frac{N}{2} \rfloor$. Then the distinct subwords of length $n$ are the subword consisting of all 0s, and the subwords that contain a 1. Since the 1 can be in any of the $n$ places, there are $n$ such subwords that contain a 1. So $p_w(n)=n+1$ for all $n \le \lfloor \frac{N}{2}\rfloor$.

Now we want to show that $K_w= \lceil\frac{N}{2}\rceil$ and $R_w=\lfloor\frac{N}{2}\rfloor$. $K_w= \lceil\frac{N}{2}\rceil$ because ``100...0" ($\lfloor\frac{N}{2}\rfloor-1$ zeroes) is the shortest suffix of multiplicity 1 in $w_N$. The minimum length $n$ such that all subwords of length $n$ have valence 1 is $n=\lfloor\frac{N}{2}\rfloor$. So $R_w=\lfloor\frac{N}{2}\rfloor$ and $K_w= \lceil\frac{N}{2}\rceil$.

Since $R_w < K_w$, it follows from Theorem \ref{main} that $p_w$ is constant on the interval $[R_w, K_w]$ and $p_w$ is decreasing by one on the interval $[K_w,N]$, so $p_w(n)=N-n+1$ for all $\lceil\frac{N}{2}\rceil \le n \le N$. 

\begin{figure}[ht]
\begin{center}
\psfragscanon
\psfrag{1}{$1$}
\psfrag{N}{$N$}
\psfrag{p}{$p_w(n)$}
\psfrag{n}{$n$}
\psfrag{a}{$\lfloor\frac{N}{2}\rfloor$}
\psfrag{b}{$\lceil \frac{N}{2} \rceil$}
\includegraphics[scale=.4]{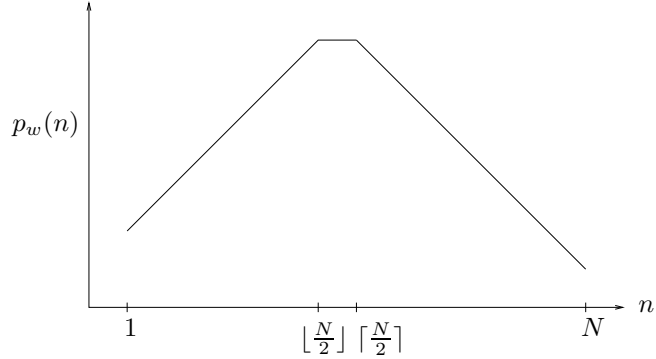}
\caption{Subword complexity sequence for $N$ odd.} 
\end{center}
\end{figure}
\end{enumerate}

Thus for any even $N$, there exists a binary word of length $N$ satisfying  $p_w(n)=n+1$ for $1 \le n \le \lfloor\frac{N}{2}\rfloor$, and $p_w(n)=N-n+1$ for $\lceil \frac{N}{2}\rceil \le n \le N$.

\end{proof}

Obviously the word above is a very low complexity word. We can show that $w_N$ is a prefix of an infinite Sturmian word, and therefore a finite Sturmian word. 

We will show that, for any $l$, the word $u=0^{l+1}10^l$ is a Sturmian word. Consider the Fibonacci word $f$ defined in Example \ref{fib:ex} and the morphism $\psi$ defined by
\begin{align*}
&\psi(0)=0^{l+1}1,\\
&\psi(1)= 0^l1.
\end{align*}

Note that $u$ is a prefix of $\psi(f)$. We will show that $\psi(f)$ is an infinite Sturmian word.

Since $f$ is aperiodic, $\psi(f)$ is aperiodic as well. We need to show that $\psi(f)$ is balanced. Assume, for contradiction, that $\psi(f)$ is unbalanced. Then there exists a subword $v$ of $\psi(f)$ such that both $0v0$ and $1v1$ are subwords of $\psi(f)$. Since $\psi(f)$ consists of $0^l1$ and $0^{l+1}1$ blocks, there exists a subword $z$ of $\psi(f)$ such that $10^l1z10^{l}1$ and $10^{l+1}1z10^{l+1}1$ are subwords of $\psi(f)$. Hence there exists a subword $x$ of $f$, $z=\psi(x)$, such that $1x1$ and $0x0$ are subwords of $f$, a contradiction to $f$ being balanced. Since $\psi(f)$ is both balanced and aperiodic, it follows from Theorem \ref{SBA:th} that $\psi(f)$ is infinite Sturmian. Therefore $u$ is finite Sturmian.

Hence $w_N$, defined in Proposition \ref{peak}, is finite Sturmian.

\bigskip
\section{Conjectures and Open Problems}

\begin{definition}
Let $a_k(n)$ denote the number of distinct subword complexity sequences of length $n \ge 1$ over a $k$-letter alphabet.
\end{definition}

\begin{table}[ht] 
\caption{Number of Distinct Subword Complexity Sequences} 
\centering     
\begin{tabular}{| c | c | c | c | c | c | c | c |}   
\hline\hline                     
n & $a_2(n)$ & $a_3(n)$ & $a_4(n)$ & $a_5(n)$ & $a_6(n)$ & $a_7(n)$ & $a_8(n)$\\ [0.5ex] 
\hline                  
1 & 1 & 1 & 1 & 1 & 1 & 1 & 1\\  \hline 
2 & 2 & 2 & 2 & 2 & 2 & 2 & 2\\ \hline 
3 & 2 & 3  & 3 & 3 & 3 & 3 & 3  \\ \hline 
4 & 3 & 4 & 5 & 5 & 5  & 5 & 5 \\ \hline 
5 & 4 & 6 & 7 & 8 & 8 & 8 & 8\\ \hline 
6 & 5 & 8 & 10 & 11 & 12 & 12 & 12 \\  \hline  
7 & 7 & 12 & 15 & 17 & 18 & 19 & 19\\ \hline 
8 & 9 & 17  & 22 & 25 & 27 & 28 & 29\\ \hline 
9 & 13 & 25 & 33 & 38 & 41 & 43 & 44\\ \hline 
10 & 18 & 37 & 49 & 57 & 62 & 65 & 67\\ \hline 
11 & 25 & 53 & 72 & 84 & 92 & 97 & 100\\   \hline  
12 & 34 & 76 & 105 & 124 & 136 & 144 & 149\\ \hline 
13 & 48 & 109  & 153 & 182 & 201 & 213 & 221\\ \hline 
14 & 67 & 159 & 224 & 268 & 297 & 316 & 328\\ \hline 
15 & 97 & 231 &  330 & 395 & 439 & 468 & 487\\ \hline 
16 & 134 & 336 & 483 &582 & 647 & 691& 720 \\  \hline 
17 & 191 & 485 & 708  &807 &906 & 1053& 1097\\ \hline 
18 & 258 & 690 & 1017  & 1164&1263 & 1362 & 1427\\ \hline 
19 & 374 & 998 & && &&\\ \hline 
20 & 521 & 1434 & & & &&\\ \hline 
21 & 738 & 2057 &  & & &&\\   \hline 
22 & 1024 &  &  & & &&\\ \hline 
23 & 1431 &   &  & & &&\\ \hline 
24 & 1972 &  & & & &&\\ \hline 
25 & 2755 &  & & &  &&\\ \hline 
26 & 3785 &  &  & & &&\\    \hline 
27 & 5244 &  &  & & &&\\ \hline 
28 & 7223 &   & & &  &&\\ \hline 
29 & 9937 &  &  & &&&\\ \hline 
30 & 13545 &  & & &&& \\ [1ex]       
\hline\hline     
\end{tabular} 
\label{table:nonlin}  
\end{table} 

\begin{conjecture}[Enayati and Green]
$a_2(n) \sim 2^{n/2}$.
\label{conj1}
\end{conjecture}

The following conjectures come from numerical data.

\begin{conjecture}
$a_k(n) \sim \log_2(k) \times 2^{n/2}$.
\end{conjecture}

\begin{table}[ht] 
\caption{Difference in the Number of Distinct Subword Complexity Sequences} 
\centering     
\begin{tabular}{| c | c | c | c | c |}   
\hline\hline                     
n & $a_3(n) - a_2(n)$ & $a_4(n) - a_3(n)$ & $a_5(n) -a_4(n)$ & $a_6(n) - a_5(n)$ \\ [0.5ex] 
\hline                  
1 & 0 & 0 & 0 & 0  \\  \hline 
2 & 0 & 0 & 0 &  0 \\ \hline 
3 & 1 & 0  & 0 & 0    \\ \hline 
4 & 1 & 1 & 0 & 0  \\ \hline 
5 & 2 & 1 & 1 & 0  \\ \hline 
6 & 3 & 2 & 1 & 1  \\  \hline  
7 & 5 & 3 & 2 & 1  \\ \hline 
8 & 8 & 5  & 3 & 2  \\ \hline 
9 & 12 & 8 & 5 & 3 \\ \hline 
10 & 19 & 12 & 8 & 5  \\ \hline 
11 & 28 & 19 & 12 & 8  \\   \hline  
12 & 42 & 29 & 19 & 12  \\ \hline 
13 & 61 & 44  & 29 & 19  \\ \hline 
14 & 92 & 65 & 44 & 29  \\ \hline 
15 & 134 & 99 &  65 & 44 \\ \hline 
16 & 202 & 147 & 99 & 65 \\  \hline 
17 & 294 & 223 &  & \\ \hline 
18 & 432 & 327 &  & \\ \hline 
19 & 624 &  & &  \\ \hline 
20 & 913 & & &  \\  \hline
21 & 1319 & & & \\ [1ex]      
\hline\hline     
\end{tabular} 
\label{dif:tb}  
\end{table}

\begin{conjecture}
There exists a function $f(k)$ such that for $n \le f(k)$: $$a_{k+2}(n)-a_{k+1}(n) = a_{k+1}(n-1)-a_k(n-1).$$
\label{con2}
\end{conjecture}
For example, $f(2)=10$. 

Note that for all $n \ge 1$, $i\ge 0$, $a_n(n)=a_{n+i}(n)$, which corresponds to a 0 in Table \ref{dif:tb}. This is because a word of length $n$ can have at most $n$ distinct letters, so adding additional letters to the alphabet will have no effect on the complexity of the word.

Also note that $a_n(n)-a_{n-1}(n)=1$ for all $n$. This is because given a word length $n$, increasing the alphabet size from $n-1$ letters to $n$ letters will only give one new word; the word containing all $n$ distinct letters. That would then give the additional subword complexity sequence $p_w = (n, n-1, n-2, \ldots, 2, 1)$. So there is an increase by 1 in the number of distinct subword complexity sequences.

Similarly, $a_{n-2}(n)-a_{n-3}(n)=2$ for all $n$. Again, given a word of length $n$, increasing the alphabet size from $n-3$ letters to $n-2$ letters will only give new words that contain all $n-2$ letters. So we know $p_w(1)=n-2$, and $p_w(4)=n-3$. Using the unimodality of subword complexity sequences we can deduce the only possible additional subword complexity sequences, and they are: $p_w = (n-2, n-1, n-2, n-3, \ldots, 2, 1)$ and $p_w = (n-2, n-2, n-2, n-3, \ldots, 2, 1)$.

We also have $a_{n-3}(n)-a_{n-4}(n)=3$ for all $n$, since the new subword complexity sequences that would result from increasing the size of the alphabet would be:
$p_w = (n-3, n-3, n-3, n-3, n-4, \ldots,1)$, $p_w = (n-3, n-2, n-2, n-3, n-4, \ldots,1)$ and $p_w = (n-3, n-1, n-2, n-3, n-4, \ldots,1)$.

It is important to note that we cannot continue this method of getting new subword complexity sequences indefinitely. We got the above sequences by using the unimodality of subword complexity sequences. There are no known necessary and sufficient conditions for subword complexity sequences, but we do know some necessary conditions and some sufficient conditions. For example, we need to consider the inequality from Theorem \ref{kineq:th}, that is, $p_w(n+1)-p_w(n) \le k(p_w(n)-p_w(n-1))$ where $k$ is the size of our alphabet, which may eliminate some of the sequences that would result using unimodality.

\section{Acknowledgements} 
I would like to thank Irina Gheorghiciuc for giving me the opportunity to write my masters thesis under her supervision. I am extremely thankful for all of her guidance, patience, and support. I would also like to thank Arda Antikacioglu for his programming help, and the committee members Tim Flaherty and James Cummings for their time and effort.

\clearpage
\bibliographystyle{plain}
\bibliography{References}
\end{document}